\documentclass{article}
\usepackage{graphicx} 
\usepackage{amsmath}
\usepackage{amssymb}
\usepackage{amsfonts}
\usepackage{amsthm}
\usepackage[margin=1in]{geometry}
\usepackage{hyperref}
\usepackage{xcolor}
\usepackage{mathrsfs}
\usepackage{tikz}
\usepackage{enumitem}
\usetikzlibrary{cd}
\usetikzlibrary{graphs}

\newcommand{\ol}{\overline}
\newcommand{\tb}{\textbf}
\newcommand{\tcb}{\textcolor{blue}}
\newcommand{\lam}{\lambda}
\newcommand{\til}{\widetilde}
\newcommand{\Lam}{\Lambda}
\newcommand{\Graphs}{{\textsf{\textup{Graphs}}}}
\newcommand{\WGraphs}{{\textsf{\textup{WGraphs}}}}
\newcommand{\bs}{\backslash}
\newcommand{\tn}{\textnormal}
\newcommand{\K}{\mathbb{K}}
\newcommand{\mc}{\mathcal}
\newcommand{\se}{\subseteq}
\newcommand{\eps}{\varepsilon}
\newcommand{\mWGraphs}{\mathfrak{m}{\mathsf{WGraphs}}}

\newtheorem{theorem}{Theorem}
\numberwithin{theorem}{section}
\newtheorem{prop}[theorem]{Proposition}
\newtheorem{lemma}[theorem]{Lemma}

\theoremstyle{definition}
\newtheorem{definition}[theorem]{Definition}
\newtheorem{example}[theorem]{Example}
\newtheorem{remark}[theorem]{Remark}

\title{Hopf algebra maps taking chromatic symmetric functions to their graph complements}
\author{Shao Yuan Lin \\ \href{mailto:syl54@cantab.ac.uk}{syl54@cantab.ac.uk} \and Laura Pierson \\ \href{mailto:lcpierson73@gmail.com}{lcpierson73@gmail.com}}

\usepackage[maxbibnames=10,style=alphabetic,maxalphanames=10]{biblatex}
\begin{document}

\maketitle

\begin{abstract}
    Cho and van Willigenburg \cite{cho2015chromatic} and Alinaeifard, Wang, and van Willgenburg \cite{aliniaeifard2021extended} introduce multiplicative \emph{\tb{\tcb{chromatic bases}}} for the ring $\Lam$ of symmetric functions, consisting of the \emph{\tb{\tcb{chromatic symmetric functions (CSFs)}}} of a sequence of connected graphs $G_1,G_2,\dots$ such that $G_n$ has total weight $n$, together with the CSFs of their disjoint unions. Tsujie \cite{tsujie2018chromatic} introduces an alternative ring structure $\til{\Lam}$ on the vector space $\Lam$ that makes CSFs multiply over joins instead of over disjoint unions. The $\til{m}_\lam$ basis, consisting of all CSFs of weighted cliques, is a multiplicative basis for $\til{\Lam}$, as is the $r_\lam$ basis of complete multipartite graphs studied in \cite{penaguiao2020kernel} and \cite{crew2021complete}. We show that one can get more of these \emph{\tb{\tcb{cochromatic bases}}} (where the starting graphs are combined by joins instead of by disjoint unions, hence forming a multiplicative basis for $\til{\Lam}$ instead of for $\Lam$) if and only if the starting graphs are edgeless. We also show that $\til{\Lam}$ is a Hopf algebra with the same coproduct as $\Lam$, and that many of the chromatic bases for $\Lam$ generated by cliques can be taken to their corresponding cochromatic bases via Hopf algebra isomorphisms $\Lam \to \til{\Lam}.$ We also show that there is a single Hopf algebra morphism $\Lam \to \til{\Lam}$ (though not an isomorphism) taking the CSFs of all unweighted triangle-free graphs to the CSFs of their complements, and we give several more conditions and examples for when one can or cannot find Hopf algebra maps from $\Lam$ to $\til{\Lam}$ taking the CSFs of certain graphs to the CSFs of their complements. Finally, we show that $K$-analogues of many of the above statements also hold if one instead uses the \emph{\tb{\tcb{Kromatic symmetric function (KSF)}}} of Crew, Pechenik, and Spirkl \cite{crew2023kromatic} in place of the CSF.
\end{abstract}

\section{Introduction}

Several of the classic vector space bases for the ring of symmetric functions $\Lambda$ (namely, the $e_\lambda$'s, $h_\lambda$'s, and $p_\lambda$'s) are \emph{\tb{\tcb{multiplicative}}} in the sense that in each case, $b_\lambda = b_{\lambda_1}\dots b_{\lambda_\ell}$ for any partition $\lambda = \lambda_1 \dots \lambda_\ell$. Equivalently, each of these bases consists of all possible products of an algebraically independent sequence of generators $b_1,b_2,\dots$ (so in each case, the sequence $b_1,b_2,\dots$ generates $\Lambda$ as an algebra). The \emph{\tb{\tcb{elementary symmetric functions}}} $e_\lambda$ and the \emph{\tb{\tcb{power sum symmetric functions}}} $p_\lambda$ can also be thought of as coming from \emph{\tb{\tcb{chromatic symmetric functions (CSFs)}}}, in the sense that $p_n$ is the CSF of a single weight $n$ vertex, and $n!e_n=X_{K_n}$ is the CSF of a clique consisting of $n$ weight 1 vertices. Multiplying CSFs corresponds to taking disjoint unions of the associated graphs, so $p_\lambda$ is the CSF of a weighted edgeless graph on $\ell$ vertices with vertex weights $\lambda_1,\dots,\lambda_\ell$, and $\lam!e_\lambda$ is the CSF of a disjoint union of $\ell$ cliques of sizes $\lambda_1,\dots,\lambda_\ell$, where $\lam! := \lam_1!\dots \lam_\ell!$.

Cho and van Willigeburg \cite{cho2015chromatic} and Alinaeifard, Wang, and van Willigenburg \cite{aliniaeifard2021extended} show that $\Lambda$ has many more multiplicative bases coming from CSFs, which they call \emph{\tb{\tcb{chromatic bases}}}. Namely, for any sequence of connected weighted graphs $G_1,G_2,\dots$ such that $G_n$ has total weight $n$, they show that the CSFs $X_{G_1},X_{G_2},\dots$ generate $\Lambda$ as an algebra, meaning the CSFs $X_{G_\lambda}$ form a multiplicative basis for $\Lambda,$ where $G_\lambda := G_{\lambda_1}\sqcup \dots \sqcup G_{\lambda_\ell}.$ 

The monomial basis $m_\lambda$ is another classic basis for $\Lambda$ that is not multiplicative but still plays an important role in the theory of symmetric functions, and specifically of CSFs. In the CSF context, the \emph{\tb{\tcb{augmented monomial symmetric functions}}} $\til{m}_\lambda$ (which are scalar multiples of the $m_\lambda$'s) are the CSFs of weighted complete graphs where the vertex weights are the parts of $\lam$. Also, the coefficient $[\til{m}_\lambda]X_G$ in the $\til{m}$-expansion of any CSF is equal to $|\tn{St}_\lam(G)|,$ the number of ways to partition the vertices into stable sets whose weights are the parts of $\lam.$ 

To make it easier to talk about ways to combine these stable set partitions (in the context of proving that the CSF distinguishes certain graphs), Tsujie \cite{tsujie2018chromatic} defines an alternative ring structure $\widetilde{\Lambda}$ on the vector space of symmetric functions such that the $\til{m}_\lam$'s become a multiplicative basis instead of the $e_\lam$'s, $h_\lam$'s, or $p_\lam$'s being multiplicative. His multiplication $\odot$ is given by $\til{m}_\lam\odot \til{m}_\mu := \til{m}_{\lam \sqcup \mu}$, where $\lam\sqcup \mu$ is the partition formed by taking all part of $\lam$ together with all parts of $\mu$ (so for instance, $322 \sqcup 421 = 432221$). Tusjie shows that $X_{G\odot H} = X_G\odot X_H$ under this multiplication, where $G\odot H$ is the \emph{\tb{\tcb{join}}} formed by connecting all vertices of $G$ to all vertices of $H$. 

Penaguiao \cite{penaguiao2020kernel} introduces a \emph{\tb{\tcb{complete multipartite basis}}} consisting of the CSFs $r_\lambda$ of unweighted complete multipartite graphs whose part sizes are the parts of $\lam$, in the context of showing that the triangular modular relation independently discovered by Guay-Paquet \cite{guay2013modular} and Orellana and Scott \cite{orellana2014graphs} generate the kernel of the map from the algebra of unweighted graphs to $\Lambda$. The $r_\lam$-basis was further studied by Crew and Spirkl in \cite{crew2021complete}, where they give transition formulas between the $r_\lam$ and $\til{m}_\lam$ bases, which they relate to the transition formulas between the $e_\lam$ and $p_\lam$ bases. They note that these transition formulas seem interesting in the context of CSFs of graph complements, since the $\til{m}_\lam$'s are the CSFs of weighted complete graphs, the $p_\lam$'s are the CSFs of their complements (weighted edgeless graphs), the $\lam!e_\lam$'s are the CSFs of disjoint unions of cliques, and the $r_\lam$'s are the CSFs of their complements (complete multipartite graphs).

\bigskip

A natural question given all of this is thus whether there are more of these \emph{\tb{\tcb{cochromatic bases}}} like the $\til{m}_\lam$'s and the $r_\lam$'s that are formed by taking the CSFs of the complements of all the $G_\lam$'s from one of the chromatic bases, or equivalently, by taking the CSFs of some sequence of graphs $\ol{G}_1,\ol{G}_2,\dots$ together with the CSFs of all joins $\ol{G_{\lam_1}}\odot \dots \odot \ol{G_{\lam_\ell}}$ of combinations of these graphs. While the chromatic bases are multiplicative bases of $\Lambda$, the cochromatic bases are instead multiplicative bases of $\til{\Lambda}.$ In \S\ref{sec:cochromatic}, we show that CSFs of the above type form a basis if and only if the $G_n$'s are weighted complete graphs (whereas to form a chromatic basis, they only need to be connected):

\begin{prop}\label{prop:cochromatic}
    For a sequence of graphs $G_1,G_2,\dots,$ the set of CSFs of the form $X_{\ol{G_\lam}} = X_{\ol{G_{\lam_1}}\odot \dots \odot \ol{G_{\lam_\ell}}}$ form a multiplicative basis for $\til{\Lambda}$ if and only if $G_n$ is a weighted clique of total weight $n$ for each $n$.
\end{prop}

Another natural question is whether this ring $\til{\Lam}$ has additional structure, and whether any of the ``graph complement maps" from $\Lambda$ to $\til{\Lambda}$ taking the CSFs of certain graphs to the CSFs of their complements preserve that structure. We show in \S\ref{sec:hopf_structure} that $\til{\Lam}$ is a Hopf algebra:

\begin{prop}\label{prop:hopf}
    $\til{\Lam}$ is a graded Hopf algebra with the same grading and coproduct as $\Lam$ and with antipode $\til{m}_n \mapsto -\til{m}_n,$ and it is isomorphic to $\Lam$ as a graded Hopf algebra via the map $p_\lam\mapsto \til{m}_\lam.$ 
\end{prop}

We also describe how the characters of $\til{\Lam}$ act on CSFs:

\begin{prop}\label{prop:characters}
    Each character $\zeta_a$ of $\til{\Lam}$ corresponds to a sequence of scalars $a_n=\zeta_a(\til{m}_n)$, and acts on CSFs by $\zeta_a(X_G) =  \sum_{\lam}|\tn{St}_\lam(G)|a_{\lam_1}\dots a_{\lam_\ell}.$ The sequence for the convolution $\zeta_a * \zeta_b$ has $n^{\text{th}}$ term $a_n+b_n.$
\end{prop}

In particular, when $a_1,a_2,\dots$ are positive integers, $\zeta_a(X_G)$ counts ways to partition $V(G)$ into stable sets and then assign a color to all vertices in each stable set, such that stable sets of weight $n$ have a particular set of $a_n$ color options. Convolution corresponds to taking the disjoint union of the allowed color sets for each stable set size. If $a_n = k$ for all $n$, this the same as the chromatic polynomial $\chi_G(k)$, except that multiple stable sets may get the same color instead of every stable set being required to get a different color.

We can also describe $\til{\Lam}$ as a quotient of the Hopf algebra ${\textsf{\textup{WGraphs}}}$ of weighted graphs:

\begin{prop}\label{prop:WGraphs_to_Lam_til}
    The map $G\mapsto X_{\ol{G}}$ is a morphism from the Hopf algebra ${\textsf{\textup{WGraphs}}}$ of weighted graphs to $\til{\Lam}$, with kernel given by elements of the form $G - G\bs e - G/e$, such that $G\bs e$ is $G$ with an edge $e$ removed, and $G/e$ is $G$ with $e$ contracted and the weights of its endpoints added to form the weight of the new vertex.
\end{prop}

This is almost the same as the chromatic relations $G - G\bs e + G/e$ generating the kernel of the $\WGraphs \to \Lambda$ map, except that the $G/e$ term has the opposite sign.

For any chromatic basis generated by a sequence of graphs $G_1,G_2,\dots$, the map $X_{G_\lam}\mapsto X_{\ol{G_\lam}}$ is multiplicative and so is automatically a morphism of algebras, but in general it need not be a morphism of Hopf algebras.  However, we show in \S\ref{sec:hopf_maps} that in certain cases it is:

\begin{prop}\label{prop:hopf_complement_maps}
    If $G_1,G_2,\dots$ is a family of connected weighted graphs such that $G_n$ has total weight $n$ and all its connected induced subgraphs are isomorphic to some $G_i$, the map $X_{G_\lam}\mapsto X_{\ol{G_\lam}}$ is a morphism of graded Hopf algebras (though not an isomorphism unless the $G_n$'s are cliques).
\end{prop}

In general, the maps from $\Lam$ to $\til{\Lam}$ sending the CSFs of some particular graphs to the CSFs of their complements need not also send the CSFs of other graphs to the CSFs of their complements, but in some cases, it turns out they do:

\begin{prop}\label{prop:clique_maps}
    For any sequence $v_1,v_2,\dots$ of positive integers with $v_n\le n$, there is an isomorphism of graded Hopf algebras from $\Lam$ to $\til{\Lam}$ sending $X_G$ to $X_{\ol{G}}$ for all weighted cliques $G$ such that for every $n$, all induced subgraphs of $G$ of weight $n$ have exactly $v_n$ vertices.
\end{prop}

For example, if $v_n = 1$ for all $n$, the map in Proposition \ref{prop:clique_maps} is the map $p_\lam \mapsto\til{m}_\lam$ from Proposition \ref{prop:hopf}, and if $v_n = n$ for all $n$, the map in Proposition \ref{prop:clique_maps} sends $\lam!e_\lam$ to $r_\lam$.

There is also another notable case where the CSFs of many graphs get sent to the CSFs of their complements by a single Hopf algebra morphism (though in this case it is not an isomorphism):

\begin{prop}\label{prop:unweighted_triangle_free_maps}
    There is a single morphism of graded Hopf algebras from $\Lam$ to $\til{\Lam}$ sending the CSFs of all unweighted triangle-free graphs to the CSFs of their complements. 
\end{prop}

We can also give a weighted generalization of Proposition \ref{prop:unweighted_triangle_free_maps}, of which Proposition \ref{prop:unweighted_triangle_free_maps} is the special case $V=\{1\},$ $E=\{2\}$, $C = \{3,4,5,\dots\}$:

\begin{prop}\label{prop:weighted_triangle_free_maps}
    For any nonoverlapping sets of positive integers $V,E,$ and $C$, there is a single graded Hopf algebra morphism taking the CSFs of all weighted triangle-free graphs such that the vertex weights all come from $V$, the edge weights (i.e. the sums of the weights of the endpoints of each edge) all come from $E$, and the total weights of connected subgraphs on 3 or more vertices all come from $C$, to the CSFs of their complements.
\end{prop}

Finally, we give a further generalization of both Propositions 1.6 and 1.8, of which Proposition 1.8 is the special case $V=C_1$, $E=C_2$, $C=C^*\sqcup C_3\sqcup C_4\sqcup\cdots$, and of which Proposition 1.6 is the special case $C_k=\{n\in\mathbb N\mid v_n=k\}$, $C^*=\emptyset$ (where $\mathbb N = \{1,2,3,\dots\}$ is the set of positive integers):

\begin{prop}\label{prop:the_culmination}
For non-overlapping sets $C^*,C_1,C_2,\dots$ of positive integers, there is a single graded Hopf algebra morphism taking the CSFs of all weighted graphs whose induced $k$-clique weights all come from $C_k$ for every positive integer $k$ and whose induced non-clique connected subgraph weights all come from $C^*$, to the CSFs of their complements.
\end{prop}

We introduce relevant background in \S\ref{sec:background} and prove the above results in \S\ref{sec:cochromatic}, \S\ref{sec:hopf_structure}, and \S\ref{sec:hopf_maps}. In \S\ref{sec:p-expansion_lemma}, we prove a key lemma that will be used for many of our proofs in \S\ref{sec:hopf_maps}. In \S\ref{sec:examples}, we give several examples showing the existence or non-existence of maps sending the CSFs of particular graphs to the CSFs of their complements. Then in \S\ref{sec:K-analogues}, we show via similar proofs that the $K$-analogues of many of the above results also hold, where the $K$-analogues involve the \emph{\tb{\tcb{Kromatic symmetric function (KSF)}}} $\ol{X}_G$ of Crew, Pechenik, and Spirkl \cite{crew2023kromatic} and Marberg's Hopf algebra interpretation of the KSF from \cite{marberg2023kromatic}. Finally, in \S\ref{sec:future}, we discuss several potential directions for future research.







\section{Background}\label{sec:background}

\subsection{Chromatic symmetric functions}

A \emph{\tb{\tcb{symmetric function}}} is a power series of bounded degree in a countably infinite set of variables $x_1,x_2,\dots$ that is invariant under any permutation of the variables. We write $\Lam$ for the ring of symmetric functions. A \emph{\tb{\tcb{partition}}} $\lam = \lam_1\dots \lam_\ell$ is a nondecreasing sequence of positive integers $\lam_1\dots \lam_\ell$, $\ell$ is its \emph{\tb{\tcb{length}}}, and $\lam_1,\dots,\lam_\ell$ are its \emph{\tb{\tcb{parts}}}. The \emph{\tb{\tcb{monomial symmetric functions}}} are $$m_\lam := \sum_{n_1,\dots,n_\ell\tn{ distinct}}x_{n_1}^{\lam_1}\dots x_{n_\ell}^{\lam_\ell},$$ and the \emph{\tb{\tcb{augmented monomial symmetric functions}}} are $$\til{m}_\lam := m_\lam \cdot \prod_{i\ge 1} r_i(\lam)!,$$ where $r_i(\lam)$ is the number of parts of $\lam$ equal to $i$. The \emph{\tb{\tcb{power sum symmetric functions}}} are $$p_\lam := \prod_{i=1}^\ell (x_1^{\lam_i} + x_2^{\lam_i} + \dots).$$ Note that $p_n = m_n = \til{m}_n = \sum_{i\ge 1}x_i^n.$ The \emph{\tb{\tcb{elementary symmetric functions}}} are $$e_\lam := \prod_{i=1}^\ell \left(\sum_{n_1,\dots,n_{\lam_i}\tn{ distinct}}x_{n_1}\dots x_{n_{\lam_i}}\right).$$ The $m_\lam$'s, $\til{m}_\lam$'s, $p_\lam$'s, and $e_\lam$'s all form bases for $\Lam.$ For partitions $\lam$ and $\mu$, we write $\lam\sqcup \mu$ for the partition formed by taking all parts of $\lam$ together with all parts of $\mu.$ A basis $b_\lam$ for $\Lam$ indexed by partitions is \emph{\tb{\tcb{multiplicative}}} if it satisfies $b_{\lam\sqcup \mu} = b_\lam b_\mu$ for all partitions $\lam$ and $\mu,$ or equivalently, if $b_\lam = b_{\lam_1}\dots b_{\lam_\ell}$ for all $\lam.$ The $p_\lam$ and $e_\lam$ bases are multiplicative while the $m_\lam$ and $\til{m}_\lam$ bases are not. 

We will write $\til{\Lam}$ for Tusjie's ring with the modified multiplication $$\til{m}_\lam \odot \til{m}_\mu := \til{m}_{\lam \sqcup \mu}$$ that makes the $\til{m}_\lam$'s into a multiplicative basis.

A \emph{\tb{\tcb{weighted graph}}} $G$ consists of a finite set of \emph{\tb{\tcb{vertices}}} $V(G)$, a set of \emph{\tb{\tcb{edges}}} $E(G)$ that are distinct unordered pairs of vertices, and a positive integer valued \emph{\tb{\tcb{weight function}}} $w$ on the vertices. Differing from some other sources, in order to make our notation less cumbersome we will by default assume that all graphs are weighted and will write $G$ instead of $(G,w)$ to mean a weighted graph, thus assuming the weight function $w$ is baked into the definition of $G$. If we call a graph \emph{\tb{\tcb{unweighted}}}, we mean that all vertices have weight 1. When we refer to the \emph{\tb{\tcb{weight of a subset}}} of $V(G)$, we mean the sum of the weights of its vertices. We write $\ol{G}$ for the \emph{\tb{\tcb{complement}}} of $G$, i.e. the graph with the same set of weighted vertices as $G$ but with edges between precisely the vertex pairs that were not connected in $G$. We write $G\sqcup H$ for the disjoint union of the weighted graphs $G$ and $H$ (which has no edges connecting vertices in $G$ to vertices in $H$), and we give its vertices the same weights as their weights in $G$ or $H$. We write $G\odot H$ for the \emph{\tb{\tcb{join}}} formed by connecting all vertices in $G$ to all vertices in $H$, again preserving the weights. For a subset $S\se V(G)$, the \emph{\tb{\tcb{induced subgraph}}} $G|_S$ is the graph with vertex set $S$ whose edges are the edges in $E(G)$ with both endpoints in $S$. A subset of $V(G)$ is \emph{\tb{\tcb{connected}}} if there is a path of edges between any two vertices in the subset. An \emph{\tb{\tcb{acyclic orientation}}} on $G$ assigns a direction to every edge such that there are no directed cycles, and a \emph{\tb{\tcb{source}}} is a vertex with no edges directed towards it. A \emph{\tb{\tcb{clique}}} or \emph{\tb{\tcb{complete graph}}} is a graph with an edge between every pair of vertices.

A \emph{\tb{\tcb{proper coloring}}} is an assignment of a positive-integer valued color to each vertex such that no two adjacent vertices get the same color. The \emph{\tb{\tcb{chromatic symmetric function (CSF)}}} is $$X_G := \sum_{\kappa}\prod_{v\in V(G)}x_{\kappa(v)}^{w(v)},$$ where $\kappa$ ranges over all proper colorings of $G$. 

For $e$ an edge of $G$, we write $G\bs e$ for the graph formed by deleting edge $e$, and $G/e$ for the graph formed by \emph{\tb{\tcb{contracting}}} edge $e$, i.e. removing $e$ and both its endpoints $u$ and $v$, and replacing them with a single vertex $uv$ of weight $w(u)+w(v)$, that is connected to all vertices that were either endpoints of $u$ or endpoints of $v$ in $G$ (without creating any multiple edges). The \emph{\tb{\tcb{contraction-deletion relation}}} of Crew and Spirkl \cite{crew2020deletion} says that $$X_G = X_{G\bs e} - X_{G/e}.$$

\subsection{Hopf algebras}

A \emph{\tb{\tcb{Hopf algebra}}} $A$ over a field $\K$ is an algebra over $\K$ that also has a \emph{\tb{\tcb{coproduct}}} $\Delta: A \to A\otimes A$, a \emph{\tb{\tcb{counit}}} $\eps:A\to \K,$ and an \emph{\tb{\tcb{antipode}}} $S:A\to A$, satisfying certain axioms. Namely, if $\Delta f = \sum f_1\otimes f_2$, the required counit property is $$\sum f_1\eps(f_2) = \sum \eps(f_1)f_2 = f,$$ \emph{\tb{\tcb{coassociativity}}} requires that $$\sum f_1\otimes (\Delta f_2) = \sum (\Delta f_1)\otimes f_2,$$ \emph{\tb{\tcb{compatibility}}} between the algebra and coalgebra structures means that the coproduct and counit are algebra homomorphisms, and the required antipode property is $$\sum f_1S(f_2) = \sum S(f_1)f_2 = \eps(f).$$ The Hopf algebra is \emph{\tb{\tcb{commutative}}} if the product is commutative, and \emph{\tb{\tcb{cocommutative}}} if $\Delta f$ stays the same under switching the order of the terms within each tensor in its sum, i.e. $\Delta f = \sum f_1\otimes f_2 = \sum f_2 \otimes f_1.$ Our Hopf algebras here will all be commutative and cocommutative, which implies that the antipode will be an algebra homomorphism from $A$ to itself. Given a compatible product and coproduct, the antipode can be recursively computed via this formula if it exists, so there can be at most one choice of antipode for a given product and coproduct. An element $f\in A$ is \emph{\tb{\tcb{primitive}}} if $\Delta f = f\otimes 1 + 1\otimes f$, in which case it follows that $S(f) = -f$.

A \emph{\tb{\tcb{morphism of Hopf algebras}}} is a map preserving the product, coproduct, unit, and counit (and the antipode will then automatically also be preserved). A Hopf algebra is \emph{\tb{\tcb{graded}}} if it can be written as $A = \bigoplus_{n\ge 0}A_n$ such that if $f\in A_n$ and $g\in A_m,$ then $fg\in A_{m+n}$, and if $\Delta f = \sum f_1\otimes f_2,$ then $f_2\in A_{n-i}$ whenever $f_1\in A_i$. A graded Hopf algebra is \emph{\tb{\tcb{connected}}} if $A_0\cong \K.$ A Hopf algebra morphism from $A=\bigoplus_{n\ge 0}A_n$ to $B = \bigoplus_{n\ge 0}B_n$ is a \emph{\tb{\tcb{morphism of graded Hopf algebras}}} if it sends $A_n$ to $B_n$ for every $n$.

Hopf algebras can be thought of categorically in terms of certain commutative diagrams, but an alternative way to think about what it means for $A$ to be a Hopf algebra is that the coproduct, counit, and antipode create a group structure on the set of \emph{\tb{\tcb{characters}}} of $A$, i.e. the set of multiplicative linear functions from $A$ to $\K$. The group operation is \emph{\tb{\tcb{convolution}}}, $$(\zeta_1 * \zeta_2)(f) := \sum \zeta_1(f_1)\zeta_2(f_2),$$ where again $\Delta f = \sum f_1\otimes f_2.$ The counit $\eps$ gives this group an identity, and the antipode gives it inverses because $\zeta \circ S$ is the inverse of $\zeta$. A \emph{\tb{\tcb{combinatorial Hopf algebra (CHA)}}} $(A,\zeta_A)$ is a graded connected Hopf algebra $A$ together with a specific choice of character $\zeta_A:A\to\K,$ and a \emph{\tb{\tcb{morphism of combinatorial Hopf algebras}}} $\phi:(A,\zeta_A)\to (B,\zeta_B)$ is a morphism of graded Hopf algebras such that $\zeta_A = \zeta_B\circ \phi.$

The \emph{\tb{\tcb{Hopf algebra of symmetric functions}}} $\Lam$ has coproduct $$\Delta p_n := p_n \otimes 1 + 1\otimes p_n$$ (or equivalently, $\Delta e_n = \sum_{i+j=n} e_i\otimes e_j$), counit $\eps(1)= 1$ and $\eps(x_i)= 0$ for all $i$, and antipode $$S(p_n) = -p_n,$$ since $p_n$ is primitive. The standard way to make $\Lam$ into a combinatorial Hopf algebra is by using the character $\zeta_\Lam:\Lam\to\K$ sending a single variable $x_1$ to 1 and the rest to 0, which implies that $\zeta_\Lam(1)=1$ and $\zeta_\Lam(p_n) = 1$ for all $n$ (and then by multiplicativity, $\zeta_\Lam(p_\lam) = 1$ for all $\lam$). This pair $(\Lam,\zeta_\Lam)$ is a \emph{\tb{\tcb{terminal object}}} in the category of commutative, cocommutative CHAs, which means that every CHA $(A,\zeta_A)$ in the category has a unique morphism to $(\Lam,\zeta_\Lam)$.

The \emph{\tb{\tcb{Hopf algebra of weighted graphs}}} $\WGraphs$ is the algebra over $\K$ freely generated by all connected graphs, with multiplication given by disjoint unions $G\sqcup H$. The coproduct is $$\Delta G := \sum_{S\sqcup T = V(G)}G|_S\otimes G|_T.$$ The unit is the empty graph, and the counit sends the empty graph to 1 and all other graphs to 0. $\WGraphs$ can be thought of as being graded either by number of vertices or by total weight, but we will generally want to think of it as being graded by weight. The CSF connection is that if one chooses the character $\zeta_{\WGraphs}$ sending edgeless graphs to 1 and all other graphs to 0, then the unique morphism $(\WGraphs,\zeta_{\WGraphs})\to (\Lam,\zeta_\Lam)$ is $$G\mapsto X_G.$$ The antipode for $\WGraphs$ can be described by Schmitt's formula from \cite{schmitt1994incidence}, $$S(G) = \sum_{S_1\sqcup \dots \sqcup S_\ell = V(G)}(-1)^{\ell}\ell!\cdot G|_{S_1}\sqcup \dots \sqcup G|_{S_\ell},$$ which a consequence of Takeuchi's more general formula from \cite{takeuchi1971free}. Humpert and Martin \cite{humpert2012incidence} also give a cancellation-free version $$S(G) = \sum_{F\in\mc{F}(G)} (-1)^{\kappa(G_{V,F})}|\mc{AO}(G/F)|\cdot G_{V,F},$$ where $\mc{F}(G)$ is the set of \emph{\tb{\tcb{flats}}} on $G$, meaning subsets $F\se E(G)$ such that $e\in F$ whenever $e\in E(G)$ is an edge between two vertices that are connected by a path of edges in $F$, $G_{V,F}$ is the subgraph of $G$ with vertex set $V(G)$ and edges set $F$, $\kappa(G_{V,F})$ is the number of connected components of $G_{V,F}$, $G/F$ is the graph formed by contracting all edges in $F$, and $\mc{AO}(G/F)$ is the set of acyclic orientations on $G/F.$

\section{Proof of Proposition \ref{prop:cochromatic}: Cochromatic bases}\label{sec:cochromatic}

To show that the only cochromatic bases are the ones coming from complements of cliques, is known from Tsujie \cite{tsujie2018chromatic} that in $\til{\Lam}$, CSFs multiply over joins under the $\odot$ multiplication, so $$X_{\ol{G_{\lam_1}}\odot \dots \odot \ol{G_{\lam_\ell}}} = X_{\ol{G_{\lam_1}}}\odot \dots \odot X_{\ol{G_{\lam_\ell}}}.$$ Thus, the CSFs $X_{\ol{G_\lam}}$ automatically forms a multiplicative basis for $\til{\Lam}$ as long as they form a basis, and more generally, the subspace of $\til{\Lam}$ spanned by the $X_{\ol{G_\lam}}$'s is equivalent to the subalgebra of $\til{\Lam}$ generated by the $X_{\ol{G_n}}$'s. Since the $\til{m}_n$'s are algebraically independent generators for $\til{\Lam}$, the $X_{\ol{G_n}}$'s are algebraically independent generators if and only if every $\til{m}_n$ can be written as a polynomial in the $X_{\ol{G_i}}$'s, and also there are no nontrivial polynomial relationships between the $X_{\ol{G_n}}$'s. 

We know that $[\til{m}_\lam]X_{\ol{G_n}}=|\tn{St}_\lam(\ol{G_n})|,$ the number of partitions of $V(G_n)$ into stable sets whose weights are the parts of $\lam$. Thus, for every nonzero $\til{m}_\lam$ term in $X_{\ol{G_n}}$, the size $|\lam|$ must equal the sum of the vertex weights in $G_n$. Since $\til{m}_n$ cannot be written as a polynomial in smaller $\til{m}_i$'s in $\til{\Lam}$, the only way to get $\til{m}_n$ as a polynomial in the $X_{\ol{G_i}}$'s is if one of them has a $\til{m}_n$ term, so some $G_{i_n}$ must have weight $n$. This automatically forces the graphs $G_{i_1},G_{i_2},\dots$ to be different from each other, since they all have different total weights. Also, in order for $X_{\ol{G_{i_n}}}$ to actually have a $\til{m}_n$ term, the full set of vertices of $\ol{G_{i_n}}$ must form a stable set, meaning it must be edgeless, or equivalently, $G_{i_n}$ must be a clique of total weight $n$. 

Then it follows by induction on $n$ that every $\til{m}_n$ is a polynomial in the $X_{\ol{G_{i_j}}}$'s. To see that, note that if we assume inductively that every $\til{m}_j$ is a polynomial in $X_{\ol{G_{i_1}}},X_{\ol{G_{i_2}}},\dots,X_{\ol{G_{i_j}}}$ for all $j<n,$ then all terms of $X_{\ol{G_{i_n}}}$ except the $\til{m}_n$ term are products of the $\til{m}_j$'s for $j<n,$ and hence are polynomials in $X_{\ol{G_{i_1}}},\dots,X_{\ol{G_{i_{n-1}}}},$ so subtracting all those polynomials from $X_{\ol{G_{i_n}}}$ expresses $\til{m}_n$ as a polynomial in $X_{\ol{G_{i_1}}},\dots,X_{\ol{G_{i_n}}}.$ Thus, since the $\til{m}_n$'s generate $\til{\Lam}$ as an algebra, the $X_{\ol{G_{i_n}}}$'s do as well.

If we had any CSF $X_{\ol{G_n}}$ in the list such that $n$ is not one of the $i_j$'s, then $X_{\ol{G_n}}$ would be expressible as a polynomial in the $X_{\ol{G_{i_j}}}$'s, so the set of $X_{\ol{G_n}}$'s would not be algebraically independent. Thus, we may assume $i_n = n$ for all $n$, so $G_n$ is a clique of weight $n$ for every $n$. 

The algebraic independence of any such sequence $G_1,G_2,\dots$ follows from the fact that $X_{\ol{G_n}}$ contains a $\til{m}_n$ term while no $X_{\ol{G_{i_j}}}$ for $j<n$ does. To see that, note that if we had a minimal degree polynomial relation between the $X_{\ol{G_{i_j}}}$'s with $X_{\ol{G_{i_n}}}$ as the highest weight term showing up, then for the maximal power $k$ such that $(X_{\ol{G_{i_n}}})^k$ divides some of the terms in the polynomial, all terms divisible by $(X_{\ol{G_{i_n}}})^k$ would need to cancel since each such term has a $\til{m}_n^k$ factor while no other terms do. Factoring out $(X_{\ol{G_{i_n}}})^k$ from all such terms would then give a lower degree polynomial relation between $X_{\ol{G_{i_1}}},\dots,X_{\ol{G_{i_{n-1}}}},$ contradicting minimality of the degree of the polynomial. Thus, as long as each $G_n$ is a clique of weight $n$, the $X_{\ol{G_n}}$'s are algebraically independent generators for $\til{\Lam}$ and hence the $X_{\ol{G_\lam}}$'s are a multiplicative basis for $\til{\Lam}$, as claimed. \qed

\section{Hopf algebra structure of \texorpdfstring{$\til{\Lam}$}{Lambda tilde}}\label{sec:hopf_structure}

\subsection{Proof of Proposition \ref{prop:hopf}: \texorpdfstring{$\til{\Lam}$}{Lambda tilde} as a Hopf algebra}\label{subsec:hopf}

To show that $p_\lam\mapsto\til{m}_\lam$ is a morphism of graded Hopf algebras from $\Lam\to\til{\Lam}$, first note that since it is known that the map $G\mapsto X_G$ is a morphism of Hopf algebras from $\WGraphs$ to $\Lam$, the coproduct on $\Lam$ acts on the CSF of any weighted graph $G$ by $$\Delta X_G = \sum_{S\sqcup T=V(G)}X_{G|_S}\otimes X_{G|_T}.$$ In particular, since $p_\lam$ is the CSF of a weighted edgeless graph with weights the parts of $\lam,$ we can write $$\Delta p_\lam = \sum_{\mu\sqcup \nu = \lam} a_{\mu\nu}^\lam (p_\mu \otimes p_\nu),$$ where the scalar $a_{\mu\nu}^\lam$ counts the number of ways to split the parts of $\lam$ into two subsets such that the parts in the first subset form the partition $\mu$ and the parts in the second set form the partition $\nu$. (We count repeated parts of $\lam$ of the same size as being different from each other in this context, since different vertices of the graph count as different even if they have the same weight.) Since the starting graph is a weighted edgeless graph, all of its induced subgraphs are also weighted edgeless graphs, so the terms above are the only ones we get in $\Delta p_\lam.$

By similar reasoning, since $\til{m}_\lam$ is the CSF of a weighted complete graph with vertex weights the parts of $\lam$, all its induced subgraphs are of the form $\til{m}_\mu$ where the parts of $\mu$ are a subset of the parts of $\lam,$ and the terms $\til{m}_\mu\otimes \til{m}_\nu$ in $\Delta \til{m}_\lam$ correspond to ways to split the parts of $\lam$ between $\mu$ and $\nu$, if repeated parts of $\lam$ are considered different from each other. Thus we get $$\Delta \til{m}_\lam = \sum_{\mu\sqcup \nu = \lam} a_{\mu\nu}^\lam(\til{m}_\mu\otimes \til{m}_\nu),$$ where the structure constants $a_{\mu\nu}^\lam$ are the same as for $\Delta p_\lam$.

Now give $\til{\Lam}$ the same coproduct as $\Lam$ and the antipode map $\til{m}_n \mapsto -\til{m}_n$, and consider the linear map $p_\lam \mapsto \til{m}_\lam$ from $\Lam$ to $\til{\Lam}$. This map is invertible, since the $p_\lam$'s and $\til{m}_\lam$'s are both vector space bases for $\Lam,$ and $\til{\Lam}$ is the same as $\Lam$ as a vector space. The map is also an algebra homomorphism, since the $p_\lam$'s are a multiplicative basis for $\Lam$ and the $\til{m}_\lam$'s are a multiplicative basis for $\til{\Lam}$, so $p_\lam p_\mu = p_{\lam\sqcup \mu} \mapsto \til{m}_{\lam\sqcup\mu} = \til{m}_\lam \odot \til{m}_\mu.$ The coproduct on $\Lam$ also has the same structure constants $a_{\mu\nu}^\lam$ with respect to the $p_\lam$ basis as the coproduct on $\til{\Lam}$ has with respect to the $\til{m}_\lam$ basis, and the coproduct in each case is entirely determined through linearity by its action on the basis vectors. The antipode also has the same structure constants, since the antipode on $\Lam$ is $p_n\mapsto -p_n.$

Thus, since the product, coproduct, and antipode are preserved, the $p_\lam\mapsto \til{m}_\lam$ map is a morphism of Hopf algebras as long as $\til{\Lam}$ is a valid Hopf algebra. But $\til{\Lam}$ being a valid Hopf algebra is now immediate from $\Lam$ being a valid Hopf algebra, because the requirements for being a Hopf algebra can all be expressed as certain relationships between the structure constants for multiplication, comultiplication, and the antipode. Hence, since the $p_\lam \mapsto \til{m}_\lam$ map is an invertible linear map from $\Lam$ to $\til{\Lam}$ that preserves all these structure constants, and $\Lam$ is known to be a valid Hopf algebra, it automatically follows that $\til{\Lam}$ is also a valid Hopf algebra, and that the map is an isomorphism of Hopf algebras. 

We can also see that this map preserves the grading, since the grading of both $\Lam$ and $\til{\Lam}$ is by the degree of the polynomials, and the polynomials $p_\lam$ and $\til{m}_\lam$ are homogeneous of the same degree $|\lam|$ for any $\lam$. \qed

\begin{remark}
    We can also make $\til{\Lam}$ isomorphic to $\Lam$ as a CHA by choosing the character $$\zeta_{\til{\Lam}}(\til{m}_\lam):= 1$$ for all $\lam,$ since $\zeta_\Lam(p_\lam):=1$ is the standard character used to make $\Lam$ into a CHA. Then since $(\Lam,\zeta_\Lam)$ is a terminal object in the category of commutative and cocommutative CHA's, so is $(\til{\Lam},\zeta_{\til{\Lam}}).$ We can then associate a character to any Hopf algebra map we define from another Hopf algebra to $\til{\Lam}$ by composing $\zeta_{\til{\Lam}}$ with the Hopf algebra map.
\end{remark}

\subsection{Proof of Proposition \ref{prop:characters}: Characters of \texorpdfstring{$\til{\Lam}$}{Lambda tilde}}

The characters of $\til{\Lam}$ are by definition the multiplicative linear maps from $\til{\Lam}$ to $\K$. Such a map is uniquely determined by its action on the generators $\til{m}_n$ for $\til{\Lam}$ as an algebra, so each character $\zeta_a$ is uniquely determined by the sequence $a_1,a_2,\dots\in \K$ such that $\zeta_a(\til{m}_n) = a_n.$ Since the $\til{m}_\lam$-expansion for the CSF of any weighted graph $G$ is $X_G = \sum_{\lam}|\tn{St}_\lam(G)|\til{m}_\lam$, it follows immediately from linearity and multiplicativity that $$\zeta_a(X_G) = \sum_\lam |\tn{St}_\lam(G)|a_{\lam_1}\dots a_{\lam_\ell},$$ as claimed in Proposition \ref{prop:characters}. For the convolution of two characters, since $\Delta \til{m}_n = \til{m}_n\otimes 1 + 1\otimes \til{m}_n$, it follows that if $\zeta_a$ corresponds to the sequence $a_1,a_2,\dots$ and $\zeta_b$ to the sequence $b_1,b_2,\dots,$ then $$(\zeta_a * \zeta_b)(\til{m}_n) = \zeta_a(\til{m}_n)\zeta_b(1) + \zeta_a(1)\zeta_b(\til{m}_n) = a_n+b_n,$$ as claimed. \qed

\begin{remark}
    Note that the character $\zeta_{\til{\Lam}}$ defined in \S \ref{subsec:hopf} corresponds to the sequence $a_n = 1$ for all $n$, and $\zeta_{\til{\Lam}}(X_G)$ counts all ways to partition $V(G)$ into stable sets. The sequence given by $a_n=k$ for all $n$ is the $k$-fold convolution of $\zeta_{\til{\Lam}}$ with itself, and counts the ways to partition $V(G)$ into stable sets and then assign one of $k$ colors to each stable set.
\end{remark}

\subsection{Proof of Proposition \ref{prop:WGraphs_to_Lam_til}: \texorpdfstring{$\til{\Lam}$}{Lambda tilde} as a quotient of \texorpdfstring{$\WGraphs$}{WGraphs}}

Since we know that $\WGraphs$ and $\til{\Lam}$ are both valid Hopf algebras, to show that $G\mapsto X_{\ol{G}}$ is a morphism of Hopf algebras, it suffices to show that it preserves both the product and coproduct. 

The map preserves the product because the product in $\WGraphs$ is $G\sqcup H$, and $$X_{\ol{G\sqcup H}} = X_{\ol{G}\odot\ol{H}} = X_{\ol{G}}\odot X_{\ol{H}},$$ where the first equality follows because the complement of the disjoint union of two graphs is the same as the join of their complements, and the second equality follows because we know from Tsujie \cite{tsujie2018chromatic} that CSFs multiply over joins in $\til{\Lam}.$

To see that the coproduct is preserved, recall that the coproduct on $\WGraphs$ is 
\begin{equation}\label{eqn:graph_coproduct}
    \Delta G = \sum_{S\sqcup T = V(G)} G|_S\otimes G|_T.
\end{equation} 
Since the map $\WGraphs \to \Lam$ given by $G\mapsto X_G$ is known to be a Hopf algebra morphism, in $\Lam$ we have $$\Delta X_G = \sum_{S\sqcup T = V(G)} X_{G|_S}\otimes X_{G|_T},$$ and the above equality also holds in $\til{\Lam}$, since we defined the coproduct on $\til{\Lam}$ to be the same as the coproduct on $\Lam.$ The same equation holds if we use $\ol{G}$ in place of $G$: $$\Delta X_{\ol{G}} = \sum_{S\sqcup T = V(\ol{G})} X_{\ol{G}|_S} \otimes X_{\ol{G}|_T}.$$ Now the key fact is that each induced subgraph of the complement $\ol{G}$ is the complement of the corresponding induced subgraph of $G$, i.e. $\ol{G}|_S = \ol{G|_S}.$ Thus, we can equivalently write 
\begin{equation}\label{eqn:lam_tilde_coproduct}
    \Delta X_{\ol{G}} = \sum_{S\sqcup T=V(G)} X_{\ol{G|_S}}\otimes X_{\ol{G|_T}}.
\end{equation} 
But now the terms of (\ref{eqn:lam_tilde_coproduct}) are precisely the images under our $G\mapsto X_{\ol{G}}$ map of the corresponding terms from (\ref{eqn:graph_coproduct}), hence the coproduct is preserved, so the map is a morphism of Hopf algebras.

For the kernel, it is known that the contraction-deletion relation $X_G = X_{G\bs e}-X_{G/e}$ of Crew and Spirkl \cite{crew2020deletion} for CSFs of weighted graphs can be used to get between any two graphs with equal CSF, since it can be used to reduce the CSF of any graph to its $p$-expansion. Now note that since $e$ is not an edge of $G\bs e$, $e$ is an edge of $\ol{G\bs e}$. Also, $\ol{G} = (\ol{G\bs e})\bs e$, since $\ol{G}$ is the same as $\ol{G\bs e}$ except that $\ol{G}$ does not contain the edge $e$ while $\ol{G\bs e}$ does. Similarly, $(\ol{G\bs e})/e = \ol{G/e}$, since both $\ol{G\bs e}$ and $G$ have $e$ as an edge to start (meaning it can be contracted), contracting it gives new graphs with the same vertex set, and the resulting edges of the new graphs are the same. Thus, plugging in $\ol{G\bs e}$ in place of $G$ in the contraction-deletion relation gives $$X_{\ol{G\bs e}} = X_{(\ol{G\bs e})\bs e} - X_{(\ol{G\bs e})/e} = X_{\ol{G}}-X_{\ol{G/e}}.$$ Moving the $X_{\ol{G\bs e}}$ term to the right, it follows that $G - G\bs e - G/e$ is in the kernel of our $G\mapsto X_{\ol{G}}$ map. To see that these relations generate the kernel, we recall that $\til{\Lam}$ is equivalent to $\Lam$ as a vector space, so since the set of linear relations $X_G=X_{G\bs e}-X_{G/e}$ are enough to reduce the CSF of any graph to its $p$-expansion, the equivalent set of linear relations $X_{\ol{G\bs e}}=X_{\ol{G}}-X_{\ol{G/e}}$ are also enough to reduce any graph to its $p$-expansion, so the relations $G-G\bs e - G/e$ generate the kernel. \qed

\begin{remark}
    In terms of characters, the character we get on $\WGraphs$ by composing $\zeta_{\til{\Lam}}$ with this map is the map sending $G$ to the number of ways to partition $V(G)$ into cliques.
\end{remark}

\section{A weighted power sum expansion using acyclic orientations}\label{sec:p-expansion_lemma}

For our next few proofs, we will make use of a weighted analogue of the $p_\lam$-expansion for CSFs from \cite{bernardi2020combinatorial}:

\begin{lemma}[cf. Bernardi-Nadeau \cite{bernardi2020combinatorial}, Proposition 5.3]\label{lem:source_components}
    Given a fixed ordering on $V(G)$,
    $$X_G = (-1)^{|V(G)|}\sum_{\gamma \in\mc{AO}(G)}(-1)^{\ell(\lambda(\gamma))}p_{\lam(\gamma)},$$ where $\mc{AO}(G)$ is the set of acyclic orientations on $G,$ and $\lam(\gamma)$ is the partition whose parts are the weights of the source components of $\gamma.$
\end{lemma}

The source components are defined as follows:

\begin{definition}[Bernardi-Nadeau \cite{bernardi2020combinatorial}]
    Given a fixed order on $V(G)$, the \emph{\tb{\tcb{source components}}} of an acyclic orientation $\gamma$ are recursively constructed as follows, where we start with all vertices being unused:
\begin{enumerate}
    \item Take the first unused vertex to be the source of the new source component (and mark it as used).
    \item Add all unused vertices reachable from the source by directed paths to its source component.
    \item Repeat until all vertices are assigned to a source component.
\end{enumerate}
(It is not obvious that the multiset of partitions $\lam(\gamma)$ is independent of the choice of ordering on $V(G)$, but it turns out that it is.)
\end{definition}

Since Bernardi and Nadeau only prove Lemma \ref{lem:source_components} for unweighted graphs but we will want to use the weighted version, we will show how to modify their argument to work for weighted graphs. Some of our notation will be slightly different from theirs.

Following \cite{bernardi2020combinatorial}, we will use the language of the graphical interpretation of the \emph{\tb{\tcb{heaps of pieces}}} of Cartier and Foata \cite{cartier2006problemes} and Viennot \cite{viennot2006heaps}. A \emph{\tb{\tcb{heap}}} on $G$ is formed by blowing up each vertex of $G$ into a clique (possibly empty and possibly a single vertex) such that two vertices are adjacent in the heap if they come from either the same or adjacent vertices of $G$. Repeated vertices in a heap are not considered distinguishable from each other, in the sense that two heaps are not considered different if they can be obtained from each other by permuting repeated copies of the same vertex. 
    
A \emph{\tb{\tcb{pyramid}}} is a heap with a unique source, and we write $\mathscr{P}$ for the set of pyramids on $G$. If $G$ has $n$ vertices and we assign a sequence of variables $\mathbf{v}=(v_1,\dots,v_n)$ to the vertices, then $$P(\mathbf{v}) := \sum_{\tb{h}\in\mathscr{P}}\frac{\mathbf{v}^{\mathbf{h}}}{|\mathbf{h}|}$$ is the generating series for all pyramids on $G$, where $\mathbf{v}^{\mathbf{h}}$ denotes the monomial such that the exponent on $v_i$ is the number of times vertex $v_i$ appears in the pyramid $\mathbf{h}$, and $|\mathbf{h}|$ is the total number of vertices in $\mathbf{h}$, counted with multiplicity. 
    
A \emph{\tb{\tcb{trivial heap}}} is a heap with no edges, and we write $\mathscr{T}$ for the set of trivial heaps on $G$. Since repeated copies of a vertex are automatically connected to each other, a trivial heap can contain at most one copy of each vertex, and the vertices in contains cannot have any edges between them, so a trivial heap is the same thing as a stable set. The generating series for trivial heaps is $$T(\mathbf{v}) := \sum_{\mathbf{h}\in \mathscr{T}}\mathbf{v}^{\mathbf{h}},$$ where $\mathbf{v}^{\mathbf{h}}$ denotes the monomial where the exponent on $v_i$ is the number of copies of vertex $i$ in the heap. 

\begin{proof}[Proof of Lemma \ref{lem:source_components}]

    If $G$ has weight function $w,$ its CSF is equal to $$X_G(x_1,x_2,\dots) = [v_1\dots v_n]\prod_{i\ge 1}T(x_i^{w(v_1)}v_1, x_i^{w(v_2)}v_2,\dots,x_i^{w(v_n)}v_n),$$ because $T(x_i^{w(v_1)}v_1, x_i^{w(v_2)}v_2,\dots,x_i^{w(v_n)}v_n)$ is the generating series for ways to assign color $i$ to all vertices in some stable set, where the exponent on $x_i$ is the sum of the weights of the vertices it is assigned to, and then multiplying over colors and taking the coefficient of $v_1\dots v_n$ corresponds to assigning exactly one color to each vertex. Now, Viennot shows in \cite{viennot2006heaps} that $$T(\mathbf{v}) = \exp(-P(-\mathbf{v})).$$ Plugging that in gives 
    \begin{align*}
        X_G &= [v_1\dots v_n]\prod_{i\ge 1}\exp(-P(-x_i^{w(v_1)}v_1,\dots,-x_i^{w(v_n)}v_n)) \\
        &= [v_1\dots v_n]\exp\left(-\sum_{i\ge 1}P(-x_i^{w(v_1)}v_1,\dots,-x_i^{w(v_n)}v_n)\right).
    \end{align*}
    For a subset $S\subseteq V(G),$ write $\mathscr{P}_S$ for the set of pyramids on $G$ using only the vertices in $S$ and using them each exactly once, and $$P_S(\tb{v}):= \sum_{\mathbf{h}\in \mathscr{P}_S}\frac{\mathbf{v}^{\mathbf{h}}}{|\mathbf{h}|} = \frac{|\mathscr{P}_S|}{|S|}\prod_{v\in S}v$$ for the generating series for those pyramids. These are the only sorts of pyramids that can contribute to a $v_1\dots v_n$ term (since all other pyramids have repeated vertices), so we can rewrite our desired coefficient as 
    \begin{align*}
         X_G &= [v_1\dots v_n]\prod_{S\se V(G)}\exp\left(-\sum_{i\ge 1}P_S(-x_i^{w(v_1)}v_1,\dots,-x_i^{w(v_n)}v_n)\right) \\
        &= [v_1\dots v_n]\prod_{S\se V_G}\exp\left(-\frac{|\mathscr{P}_S|}{|S|}(-1)^{|S|}p_{w(S)}\prod_{v\in S}v\right).
    \end{align*}
    To get the second line from the first, note that for all terms in $P_S(-x_i^{w(v_1)}v_1,\dots,-x_i^{w(v_n)}v_n)$, the sign is $(-1)^{|S|}$ and the power of $x_i$ is $x_i^{w(S)},$ where $|S|$ is the number of vertices in $S$ and $w(S)=\sum_{v\in S}w(v)$ is the total weight of the vertices in $S$. Then summing $x_i^{w(S)}$ over all the $x_i$ variables is equivalent to taking the power sum $p_{w(S)}.$ Expanding the exponential power series gives
    \begin{align*}
        X_G &= [v_1\dots v_n] \prod_{S\se V(G)} \left(\sum_{j\ge 1}\frac 1{j!}\left((-1)^{|S|-1}\frac{|\mathscr{P}_S|}{|S|}p_{w(S)}\prod_{v\in S}v\right)^j\right).
    \end{align*}
    Taking multiple pyramids on the same vertex set $S$ would give us duplicate copies of the vertices in $S$, so we can actually ignore all terms with $j>1$ and write $$X_G = [v_1\dots v_n]\prod_{S\se V(G)}\left(1 + (-1)^{|S|-1}\frac{|\mathscr{P}_S|}{|S|}p_{w(S)}\prod_{v\in S}v\right).$$ Now to get a $v_1\dots v_n$ term, we need to take some set of subsets $S_1,\dots, S_\ell$ with $S_1\sqcup \dots \sqcup S_\ell = V(G)$, so $$X_G = \sum_{S_1\sqcup \dots \sqcup S_\ell = V(G)}(-1)^{|S_1|+\dots+|S_\ell|-\ell}\frac{|\mathscr{P}_{S_1}|}{|S_1|}\dots \frac{|\mathscr{P}_{S_\ell}|}{|S_\ell|}p_{w(S_1)\dots w(S_\ell)}.$$ The sign is $(-1)^{|S_1|+\dots+|S_\ell|-\ell} = (-1)^{|V(G)|-\ell}=(-1)^{|V(G)|}(-1)^\ell.$ It is known that the pyramids in $\mathscr{P}_{S_i}$ are actually evenly split between having each of the $|S_i|$ vertices as the source, so $|\mathscr{P}_{S_i}|/|S_i|$ is the number of pyramids on $S_i$ with the first vertex as the source, i.e. the number of ways to make $S_i$ into a source component. Then we can uniquely associate each such set of acyclic orientations making $S_1,\dots, S_\ell$ into source components with a full acyclic orientations $\gamma$ on $G$, by directing all edges between source components from the source component with a later source toward the source component with an earlier source. Then $\lam(\gamma)=w(S_1)\dots w(S_\ell)$ is the partition whose weights are the source components of $\gamma$, and $\ell=\ell(\lam(\gamma))$ is the number of source components. Putting this together gives the power sum expansion in Lemma \ref{lem:source_components}.
\end{proof}

\section{Graph complement maps from \texorpdfstring{$\Lam$}{Lambda} to \texorpdfstring{$\til{\Lam}$}{Lambda tilde}}\label{sec:hopf_maps}

\subsection{Proof of Proposition \ref{prop:hopf_complement_maps}: Graph complement maps from \texorpdfstring{$\Lam$}{Lambda} to \texorpdfstring{$\til{\Lam}$}{Lambda tilde}}

To show that $X_{G\lam}\mapsto X_{\ol{G_\lam}}$ is a Hopf algebra map $\Lam\to\til{\Lam}$ as long as all induced subgraphs of $G_n$ are a disjoint union of $G_i$'s, the reasoning is similar to our reasoning for Propositions \ref{prop:hopf} and \ref{prop:WGraphs_to_Lam_til}. Given any graphs $G_1,G_2,\dots$ generating a chromatic basis for $\Lam$, the map $X_{G_\lam}\mapsto X_{\ol{G_\lam}}$ automatically gives a well-defined linear map $\Lam\to\til{\Lam}$, since the $X_{G_\lam}$'s form a basis for $\Lam$ (though not necessarily a surjective linear map). This linear map is also automatically an algebra homomorphism, since $$X_{G_\lam}X_{G_\mu} = X_{G_\lam \sqcup G_\mu} = X_{G_{\lam \sqcup \mu}} \mapsto X_{\ol{G_{\lam \sqcup \mu}}} = X_{\ol{G_\lam}\odot \ol{G_\mu}} = X_{\ol{G_\lam}} \odot X_{\ol{G_\mu}}.$$ It also respects the grading, since the grading is by degree, and $X_{G_\lam}$ and $X_{\ol{G_\lam}}$ are homogeneous polynomials of the same degree $|\lam|$ for each $\lam.$

To see that it is also a coalgebra homomorphism, note that by our assumption that all induced subgraphs of every $G_n$ are disjoint unions of graphs from the family, it also follows that all induced subgraphs of every $G_\lam$ are of the form $G_\mu$ for some $\mu$, since an induced subgraph of $G_\lam$ is the disjoint union of induced subgraphs of each of the $G_{\lam_i}$'s. Thus, we can write each coproduct $\Delta X_{G_\lam}$ in terms of CSFs of graphs from the family: $$\Delta X_{G_\lam} = \sum_{\mu,\nu} a_{\mu\nu}^\lam(X_{G_\mu}\otimes X_{G_\nu}),$$ where the structure constant $a_{\mu\nu}^\lam$ counts the ways to number of ways to partition the vertex set of $G_\lam$ into two sets $S$ and $T$ with $S\sqcup T = V(G_\lam)$ such that $(G_\lam)|_S = G_\mu$ and $(G_\lam)|_T = G_\nu.$ Note that unlike for $\Delta p_\lam$ and $\Delta \til{m}_\lam$, we will not necessarily have $\lam = \mu\sqcup \nu$ for all terms $G_\mu \otimes G_\nu$ showing in the coproduct, since each $G_{\lam_i}$ may be a graph with multiple vertices and so may be split between $G_\mu$ and $G_\nu$ instead of being a single vertex of weight $\lam_i$ that must be assigned to either $G_\mu$ or $G_\nu$.

Then, since the complement $\ol{G_\mu}$ of any induced subgraph $G_\mu$ of $G_\lam$ is isomorphic to the corresponding induced subgraph of the complement $\ol{G_\lam}$, whenever we have a vertex partition $S\sqcup T = V(G_\lam)$ such that $(G_\lam)|_S = G_\mu$ and $(G_\lam)|_T = G_\nu,$ we will also have $(\ol{G_\lam})|_S = \ol{G_\mu}$ and $(\ol{G_\lam})|_T = \ol{G_\nu}.$ Thus, we can write $$\Delta X_{\ol{G_\lam}} = \sum_{\mu,\nu}a_{\mu\nu}^\lam(X_{\ol{G_\mu}}\otimes X_{\ol{G_\nu}})$$ with the exact same structure $a_{\mu\nu}^\lam$ as for $\Delta X_{G_\lam}$. Thus, the coproduct is preserved for the basis vectors and hence for all linear combinations of them, so the map is an isomorphism of graded Hopf algebras. \qed

\begin{remark}
    The character on $\Lam$ associated to this map would be the one sending $X_{G_\lam}$ to the number of ways to partition $V(G_\lam)$ into cliques for all graphs $G_\lam$ in our family.
\end{remark}

\subsection{Proof of Proposition \ref{prop:clique_maps}: Graph complement maps for cliques}\label{subsec:clique_maps_proof}

To prove Proposition \ref{prop:clique_maps}, we claim that the algebra homomorphism $\Lam \to \til{\Lam}$ given by $$p_n \mapsto \frac{(-1)^{v_n-1}}{(v_n-1)!}\til{m}_n$$ sends $X_G$ to $X_{\ol{G}}$ for every clique $G$ such that all its induced subgraphs of weight $n$ have exactly $v_n$ vertices for each $n$. 

By Lemma \ref{lem:source_components}, the coefficient $(-1)^{|V(G)|-\ell(\lam)}[p_\lam]X_G$ counts the ways to partition the vertices into parts of weights $\lam_1,\dots,\lam_\ell$ and then choose an acyclic orientation so that those parts become source components.

Since we are assuming the vertex ordering is already fixed, we only have one choice for the source of each source component (namely, the first vertex within its component). We also have only one choice for the directions of edges between different source components, since they must always be directed from the source component with a later source toward the one with an earlier source, or else additional vertices would have been added to the source component with the earlier source. 

Thus, the only thing to choose is the orientations of the edges within each source component. Since the graph is a weighted clique, all its induced subgraphs (and in particular the source components) are also cliques. Choosing an acyclic orientation on a clique with a given source is equivalent to choosing an ordering of all the vertices in the clique except the source, since all edges from the source are required to point outward, and an acyclic orientation on the remaining vertices is equivalent to an ordering of those vertices. For a clique on $v_n$ vertices, this can be done in $(v_n-1)!$ ways.

Thus, $(-1)^{|V(G)|-\ell(\lam)}[p_\lam]X_G$ is equal to the number of ways to partition $V(G)$ into parts of weights $\lam_1,\dots,\lam_\ell$, multiplied by $(v_{\lam_1}-1)!\dots (v_{\lam_\ell}-1)!.$ The sign is $(-1)^{|V(G)|-\ell(\lam)}=(-1)^{v_{\lam_1}-1}\dots(-1)^{v_{\lam_\ell}-1}$, because $v_{\lam_1}+\dots+v_{\lam_\ell}=|V(G)|$ since every vertex of $G$ gets used by exactly one source component. Since by multiplicativity our map sends $$p_\lam \mapsto \frac{(-1)^{v_{\lam_1}-1}\dots (-1)^{v_{\lam_\ell-1}}}{(v_{\lam_1}-1)!\dots (v_{\lam_\ell}-1)!}\til{m}_\lam,$$ the coefficient of $\til{m}_\lam$ in the image of $X_G$ will simply be the number of ways to partition $V(G)$ into parts of weights $\lam_1,\dots,\lam_\ell$. But since $\ol{G}$ is edgeless and has all the same vertex weights as $G$, this coefficient is equal to $|\tn{St}_\lam(\ol{G})|$, since all vertex partitions of $\ol{G}$ into parts of weights $\lam_1,\dots,\lam_\ell$ are automatically partitions into \emph{stable} sets of those weights. Thus, the $\til{m}_\lam$-expansion of the image of $X_G$ is precisely the $\til{m}_\lam$-expansion of $X_{\ol{G}}$, so the image of $X_G$ is $X_{\ol{G}}$.

We can see that this map preserves grading, since $p_\lam$ and $\til{m}_\lam$ have the same degree, and the scalar factors $(-1)^{v_n-1}/(v_n-1)!$ do not change the degrees. The map is also invertible since all of these scalar factors are nonzero. 

Finally, it is a morphism of Hopf algebras because $\Delta p_n = p_n\otimes 1 + 1\otimes p_n$, and also $\Delta (a_n \til{m}_n) = (a_n\til{m}_n)\otimes 1 + 1\otimes (a_n\til{m}_n)$ for any scalar $a_n$, by linearity of the coproduct. Since the coproduct behaves the same on the generators $p_n$ as on their images $a_n\til{m}_n$, it also behaves the same by multiplicativity on all the $p_\lam$'s as on their images, and then by linearity it behaves the same on all elements of $\Lam$ as on their images in $\til{\Lam}$. Hence, the coproduct is preserved and the map is an isomorphism of Hopf algebras. \qed

\bigskip

\begin{remark}
    Since the scalar multiples of $p_n = \til{m}_n$ are in fact the \emph{only} primitive elements of degree $n$ in both $\Lam$ and $\til{\Lam}$, the graded Hopf algebra morphisms from $\Lam$ to $\til{\Lam}$ are in fact \emph{precisely} the maps that send each $p_n$ to a scalar multiple $a_n\til{m}_n$ of $\til{m}_n,$ and the \emph{isomorphisms} are precisely the maps of this form such that $a_n \ne 0$ for all $n$.
\end{remark}

\subsection{Proof of Proposition \ref{prop:unweighted_triangle_free_maps}: Graph complement map for triangle-free graphs}
\counterwithin{theorem}{subsection}
\setcounter{theorem}{0}

We claim that the Hopf algebra morphism from $\Lam$ to $\til{\Lam}$ sending the CSFs of all triangle-free graphs to the CSFs of their complements is the algebra homomorphism given by $$p_n\mapsto \begin{cases}
    \til{m}_1 & \tn{for }n=1, \\
    -\til{m}_2 & \tn{for }n=2, \\
    0 & \tn{for }n\ge 3.
\end{cases}$$
This map is a graded Hopf algebra morphism by the last paragraph of the previous proof. 

To see that it sends the CSFs of triangle-free graphs to the CSFs of their complements, we will check that it maps the $p_\lam$-expansion of $X_G$ to the $\til{m}_\lam$-expansion of $X_{\ol{G}}$ for every unweighted triangle-free graph $G$. 

In $\ol{G}$, there are no stable sets with 3 or more vertices, since $G$ does not contain any cliques on 3 more more vertices, so $X_{\ol{G}}$ only contains $\til{m}_\lam$ terms where all parts of $\lam$ are 1 or 2, and if $|V(G)|=n,$ then $[\til{m}_{2^k 1^{n-2k}}]X_{\ol{G}}$ counts the number of partial matchings on $G$ with exactly $k$ edges. 

By Lemma \ref{lem:source_components}, the corresponding coefficient $(-1)^{k}[p_{2^k 1^{n-2k}}]X_G$ counts the number of acyclic orientations on $G$ with $k$ source components of size 2 and the rest of size 1. The sign is $(-1)^k$ because $(-1)^{n-\ell(2^k 1^{n-2k})} = (-1)^{n-(n-k)} = (-1)^k.$ A source component of size 2 is an edge with the first vertex directed towards the second vertex, and to ensure that the set of source component is precisely a given set of nonoverlapping edges together with the remaining vertices as singletons, we just need to direct all edges from source components with later sources to source components with earlier sources, as before. Thus, acyclic orientations with $k$ source components of size 2 and the rest of size 1 are in 1-to-1 correspondence with $k$-edge partial matchings on $G$, so $(-1)^k[p_{2^k 1^{n-k}}]X_G = [\til{m}_{2^k 1^{n-2k}}]X_{\ol{G}}$.

If $G$ has connected components with 3 or more vertices, then those connected components can also be made into source components by directing all their edges away from the first vertex in the component, so $X_G$ will also have $p_\lam$ terms where $\lam$ has parts of size 3 or more. However, those terms will all get sent to 0 by our map. Only the terms of the $p_\lam$-expansion with all parts equal to 1 or 2 will get sent to nonzero $\til{m}_\lam$ terms, and as we have seen, those $p_\lam$ terms match up precisely with all of the $\til{m}_\lam$ terms of $X_{\ol{G}}$. It follows that our map sends $X_G$ to $X_{\ol{G}},$ as claimed. \qed

\bigskip

In fact, we can say something more general than Proposition \ref{prop:unweighted_triangle_free_maps}. Notice that the explanation for why $(-1)^k[p_{2^k1^{n-2k}}]X_G=[\til{m}_{2^k1^{n-2k}}]X_{\ol{G}}$ holds does not require $G$ to be triangle-free. The only reason we require $G$ to be triangle-free in Proposition \ref{prop:unweighted_triangle_free_maps} is to eliminate the monomial terms involving parts of size at least 3 arising in the triangle-containing cases, but we could simply define a restriction map on $\til{\Lam}$ that sends all those terms to zero, and our result will hold for all graphs. Thus, we get the following modification of Proposition \ref{prop:unweighted_triangle_free_maps}:
\begin{prop}\label{prop:diagram_commutes}
Define the following Hopf algebra morphisms: \begin{align*}\phi:\Lam&\rightarrow\til{\Lam}&\theta:\til{\Lam}&\rightarrow\til{\Lam}\\p_n&\mapsto\begin{cases}\til{m}_1&\text{for }n=1\\-\til{m}_2&\text{for }n=2\\0&\text{for }n\geq3\end{cases}&\til{m}_n&\mapsto\begin{cases}\til{m}_1&\text{for }n=1\\\til{m}_2&\text{for }n=2\\0&\text{for }n\geq3\end{cases}\end{align*} and let $\til{\Lam}_{1,2}$ be the common image of these two maps (i.e. the subalgebra of $\til{\Lam}$ generated by $\til{m}_1$ and $\til{m}_2$). Then $\phi(X_G)=\theta(X_{\ol{G}})$ for all $G$. In other words, the following diagram commutes, where $(\Graphs,\sqcup)$ and $(\Graphs,\odot)$ denote the algebra of unweighted graphs under the disjoint union and join operations respectively:
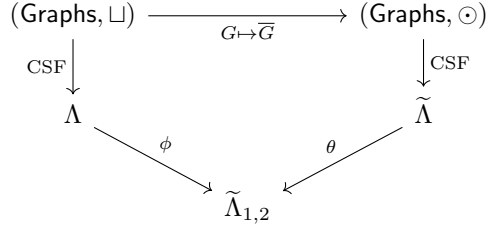
\begin{center} \begin{tikzcd}
{(\Graphs,\sqcup)} \arrow[rr, "G\mapsto\overline{G}"'] \arrow[d, "\tn{CSF}"'] &                             & {(\Graphs,\odot)} \arrow[d, "\tn{CSF}"] \\
\Lambda \arrow[rd, "\phi"]                                                              &                             & \widetilde{\Lambda} \arrow[ld, "\theta"']         \\
                                                                                        & {\widetilde{\Lambda}_{1,2}} &                                                  
\end{tikzcd} \end{center} \end{prop}

\subsection{Nonexistence of unweighted triangle-containing families except cliques}\label{subsec:nonweighted_nonexistence}
Intuitively speaking, the reason Proposition \ref{prop:unweighted_triangle_free_maps} holds is that for any partition $\lam$ whose parts have size 1 or 2, the coefficient of $p_\lam$ in $X_G$ tells us the same information as that of $\til{m}_\lam$ in $X_{\ol{G}}$. So one might naturally wonder whether this could still hold, at least for some graphs $G$, if we allow parts of size 3 and beyond. However, the following result shows that this is generally not possible:

\begin{prop}\label{prop:c_equals_zero_or_half}
Suppose that a Hopf algebra morphism $\phi:\Lam\rightarrow\til{\Lam}$ sends $p_3$ to $c\til{m}_3$, and let $F$ be a family of non-weighted connected graphs with at least one $n$-vertex graph for each positive integer $n$. Then $\phi(X_G)=X_{\ol{G}}$ for every $G$ in $F$ only if one of the following holds: \begin{itemize}
\item $c=0$ and $F$ is a family of triangle-free graphs;
\item $c=\frac12$ and $F$ is the family of cliques.
\end{itemize} \end{prop}

Proposition \ref{prop:c_equals_zero_or_half} implies that if we require the graphs to be unweighted, then the triangle-free families from Proposition \ref{prop:unweighted_triangle_free_maps} and the clique family from Proposition \ref{prop:clique_maps} with $v_n=n$ for all $n$ are the only families of graphs that can satisfy Proposition \ref{prop:hopf_complement_maps}.

\begin{proof}
First we observe that $[p_{31^{n-3}}]X_G=\#P_3(G)-\#K_3(G)$, where $\#P_3(G)$ and $\#K_3(G)$ are the number of not necessarily induced 3-vertex paths in $G$ and triangles in $G$, respectively. To see this, note that the left side counts acyclic orientations of type $(3,1^{n-3})$, which can only arise in two ways: the first is when the source component of size 3 corresponds to an induced path, and the second is when it corresponds to an induced triangle. In the former case, there is one possible way to assign an acyclic orientation (since we have to direct all other vertices away from the earliest vertex), so each induced path contributes 1 to the $[p_{31^{n-3}}]X_G$ coefficient on the left, and it also contributes 1 to the $\#P_3(G)$ term on the right. Each induced triangle contributes 2 to the left side $[p_{31^{n-3}}]X_G$, because there are 2 ways to make a triangle into a source component, as both edges from the first vertex must be directed away from that vertex, while the other edge can go in either direction. An induced triangle contains 3 (non-induced) paths and one induced triangle, so each induced triangle also contributes 2 to the right side $\#P_3(G)-\#K_3(G)$, because it contributes 3 copies of $P_3$ minus one copy of $K_3$. Summing over all connected induced 3-vertex subgraphs, we see that the two sides match.

We have $[\til{m}_{31^{n-3}}]X_{\ol{G}}=|\tn{St}_{31^{n-3}}(\ol{G})|=\#K_3(G)$, and since $\phi(p_3)=c\til{m}_3,$ in order to have $\phi(X_G) = X_{\ol{G}}$, we need $c [p_{31^{n-3}}]X_G=[\til{m}_{31^{n-3}}]X_{\ol{G}}$, so 
\begin{equation}\label{eqn:p3_m3_coeffs}
c\cdot (\#P_3(G) - \#K_3(G)) = \#K_3(G) \implies c\cdot\#P_3(G)=(c+1)\cdot\#K_3(G)
\end{equation}
needs to be satisfied for all graphs $G$ in the family $F$. In particular, this must be true for the 3-vertex member of said family. However, there are only two connected 3-vertex graphs up to isomorphism, namely the triangle and the path, whose corresponding pairs $(\#P_3(G),\#K_3(G))$ have values $(3,1)$ and $(1,0)$ respectively, corresponding to the values $c=\frac12$ and $c=0$. If $c=0$, then by (\ref{eqn:p3_m3_coeffs}), we need $\#K_3(G)=0$ for every $G\in F$, i.e. every $G\in F$ must be triangle-free. On the other hand, if $c=\frac12$, then we need $\#P_3(G)=3\cdot\#K_3(G)$ for every $G\in F$. However, since every triangle contains three 3-vertex paths, we get the inequality $\#P_3(G)\geq3\cdot\#K_3(G)$, with equality if and only if every 3-vertex path in $G$ is part of a triangle, which holds if and only if $G$ is a clique. Thus, the only two possibilities are that all graphs in $F$ are triangle-free or all graphs in $F$ are cliques, as claimed.\end{proof}

\begin{remark} The Hopf algebra morphisms described in the proofs of Propositions \ref{prop:clique_maps}  and \ref{prop:unweighted_triangle_free_maps} correspond respectively to the $c=\frac12$ and $c=0$ cases. \end{remark}

\subsection{Proof of Proposition \ref{prop:weighted_triangle_free_maps}: Maps for weighted triangle-free graphs}

To send the CSFs of our weighted triangle-free graphs with vertex weights from $V$, edge weights from $E$, and other connected subgraph weights in $C$ to the CSFs of their complements, the required map is
$$p_n \mapsto 
\begin{cases}
    \til{m}_n & \tn{for }n\in V, \\
    -\til{m}_n & \tn{for }n\in E, \\
    0 & \tn{for }n\in C.
\end{cases}$$
If there are positive integers $n \not \in V\sqcup E\sqcup C$, then we can send $p_n$ to any scalar multiple of $\til{m}_n$ we please for those values of $n.$ This map is then a morphism of graded Hopf algebras by the same argument as for Proposition \ref{prop:unweighted_triangle_free_maps} in \S\ref{subsec:clique_maps_proof}, since it sends each $p_n$ to a scalar multiple of $\til{m}_n.$

We need to show that for any weighted triangle-free graph $G$ with all its vertex weights from $V$, all its edge weights from $E$, and all its larger connected subgraph weights from $C$, the $p_\lam$-expansion for $X_G$ gets mapped to the $\til{m}_\lam$-expansion for $X_{\ol{G}}$.

In $\ol{G}$, all stable sets contain at most 2 vertices, since all cliques in $G$ contain at most 2 vertices. Thus, all stable set partitions of $\ol{G}$ have each part corresponding to either an edge of $G$ or to a single vertex of $G$. Thus, the terms in the $\til{m}_\lam$-expansion of $G$ are all of the form $\til{m}_{e_1 \dots e_k v_1 \dots v_\ell}$ for some $e_1,\dots,e_k\in E$ and $v_1,\dots,v_\ell\in V,$ and such a coefficient $[\til{m}_{e_1\dots e_k v_1\dots v_\ell}]X_{\ol{G}}$ counts the ways to partition $V(G)$ into $\ell$ singleton vertices of weights $v_1,\dots,v_\ell$ together with $k$ edges of weights $e_1,\dots,e_k.$ (We are not necessarily assuming here that $e_1\ge \dots \ge e_k \ge v_1\dots v_\ell$, so when we write $\til{m}_{e_1\dots e_k v_1\dots v_\ell}$, we mean $\til{m}_\lam$ where $\lam$ is the partition formed by arranging $e_1,\dots,e_k,v_1,\dots,v_\ell$ in decreasing order.)

The corresponding coefficient $(-1)^{k}[p_{e_1\dots e_k v_1\dots v_\ell}]X_G$ counts the ways to partition $V(G)$ into source components of weights $e_1,\dots,e_k,v_1,\dots,v_\ell.$ Since the sets $V,E,$ and $C$ do not overlap, the only way for the source components to have those sizes is if the weight $v_1,\dots,v_\ell$ source components are all singleton vertices and the weight $e_1,\dots,e_k$ source components are all edges. The sign is $(-1)^k$ because if $G$ is partitioned into $k$ edges together with $\ell$ singletons, we must have $2k+\ell=|V(G)|,$ so $$(-1)^{|V(G)|-\ell(e_1\dots e_k v_1\dots v_\ell)} = (-1)^{k+2\ell - (k+\ell)} = (-1)^k.$$ As in the proof of Proposition \ref{prop:unweighted_triangle_free_maps}, once we have chosen the set of edges and singletons, there is exactly one way to make them into source components (namely, by directing each edge that is a source component from its first vertex to its second vertex, and then directing all other edges from source components with later sources to source components with earlier sources). Thus, the coefficient $(-1)^k[p_{e_1\dots e_k v_1\dots v_\ell}]X_G$ is exactly equal to the number of $k$-edge matchings on $G$ where the edges have weights $e_1,\dots,e_k$ and the remaining singletons have weights $v_1,\dots,v_\ell,$ so it is equal to $[\til{m}_{e_1\dots e_k v_1\dots v_\ell}]X_{\ol{G}}.$

All other terms in the $p_\lam$-expansion of $X_G$ include a source component using 3 or more vertices, whose weight thus comes from $C$ since source components must be connected. Thus, all other $p_\lam$ terms of $X_G$ get sent to 0, while the ones of the form $p_{e_1\dots e_k v_1\dots v_\ell}$ get sent precisely to the $\til{m}_\lam$-terms of $X_{\ol{G}}$. It follows that $X_G$ gets mapped to $X_{\ol{G}}$, as claimed. \qed

\bigskip

The argument above essentially immediately gives us a weighted analogue of Proposition \ref{prop:diagram_commutes}, with $\til{\Lam}_{V\sqcup E}$ instead of $\til{\Lam}_{1,2}$ and the appropriate subalgebra of $\WGraphs$ instead of $\Graphs$:
\begin{prop}
Given disjoint subsets $V,E,C\se\{1,2,3,\dots\}$, define the Hopf algebra morphisms \begin{align*}\phi:\Lam&\rightarrow\til{\Lam}&\theta:\til{\Lam}&\rightarrow\til{\Lam}\\p_n&\mapsto\begin{cases}\til{m}_n&\text{for }n\in V\\-\til{m}_n&\text{for }n\in E\\0&\text{for }n\in C\end{cases}&\til{m}_n&\mapsto\begin{cases}\til{m}_n&\text{for }n\in V\\\til{m}_n&\text{for }n\in E\\0&\text{for }n\in C\end{cases}\end{align*} and let $\til{\Lam}_{V\sqcup E}$ be the common image of these two maps (i.e. the subalgebra of $\til{\Lam}$ generated by the $\til{m}_n$'s for $n\in V\sqcup E$). Then $\phi(X_G)=\theta(X_{\ol{G}})$ for all weighted graphs $G$ whose vertex weights all come from $V$, edge weights from $E$ and connected induced subgraph weights with 3 or more vertices from $C$. In other words, if $\WGraphs_{V,E,C}$ denotes the subalgebra of all such weighted graphs, then the following diagram commutes.
\begin{center} \begin{tikzcd}
{(\WGraphs_{V,E,C},\sqcup)} \arrow[rr, "G\mapsto\overline{G}"'] \arrow[d, "\tn{csf}"'] &                             & {(\WGraphs_{V,E,C},\odot)} \arrow[d, "\tn{csf}"] \\
\Lambda \arrow[rd, "\phi"]                                                              &                             & \widetilde{\Lambda} \arrow[ld, "\theta"']         \\
                                                                                        & {\widetilde{\Lambda}_{V\sqcup E}} &                                                  
\end{tikzcd} \end{center} \end{prop}

\subsection{Proof of Proposition \ref{prop:the_culmination}: Maps for given clique weights}\label{subsec:culmination_proof}
Our desired map sending all graphs $G$ with all size $k$ clique weights from $C_k$ and all non-clique connected subgraph weights from $C^*$ is $$p_n\mapsto\begin{cases}\dfrac{(-1)^{k-1}}{(k-1)!}\til{m}_n&\text{for }n\in C_k,\\[8pt]0&\text{for }n\in C^*.\end{cases}$$ For any $n\in \mathbb{N}\bs (C^*\sqcup C_1 \sqcup C_2 \sqcup \dots),$ we can send $p_n$ to any scalar multiple $a_n\til{m}_n.$

Again, we shall show that the $p_\lam$-expansion of $X_G$ maps to the $\til{m}_\lam$-expansion of $X_{\ol{G}}$. For each integer partition $\lambda\vdash\tn{weight}(G)$, let $V^\lam(G)$ denote the collection of all set partitions $\pi$ of the vertex set $V(G)$ into $\ell(\lam)$ connected parts such that the weight of the part $\pi_i$ is $\lam_i$. Recall that $$[\til{m}_\lam]X_{\ol{G}}=|\tn{St}_\lam(\ol{G})|=\#\{\pi\in V^\lam(G)\mid\pi\text{ is a stable partition of }\ol{G}\},$$ and observe that an induced subgraph of $G$ is a clique if and only if the induced subgraph on the same vertex set in $\ol{G}$ is a stable set. Hence $\pi\in V^\lam(G)$ is a stable partition of $\ol{G}$ if and only if the induced subgraphs $G|_{\pi_i}$ of $G$ are cliques, and so $[\til{m}_\lam]X_{\ol{G}}=|S^\lam(G)|$, where $$S^\lam(G)=\{\pi\in V^\lam(G)\mid G|_{\pi_i}\text{ is a clique for all }i\}.$$ Now if $\lam_i\in C^*$, then $G|_{\pi_i}$ is not a clique for any $\pi\in V^\lam(G)$, and if no $\lam_i$'s are in $C^*$, then for each $i$ we either have $\lam_i\in C_{k_i}$ for some $k_i$ or $\lam_i\in\mathbb N\setminus(C^*\sqcup C_1\sqcup C_2\sqcup\cdots)$. If there exists some $i$ where the latter holds, then $V^\lam(G)=\emptyset$ (and therefore $S^\lam(G)=\emptyset$ too) since there is no connected component of $G$ with weight $\lam_i$. Otherwise, every $\pi\in V^\lam(G)$ has $G|_{\pi_i}=K_{k_i}$, whence $\pi\in S^\lam(G)$ too. As such, $$S^\lam(G)=\begin{cases}V^\lam(G)&\text{if } \lam_i\notin C^*\text{ for all }i,\\\emptyset&\text{if }\lam_i\in C^*\text{ for some }i.\end{cases}$$ Recall that the only graded Hopf algebra morphisms $\Lam\rightarrow\til{\Lam}$ are those where $p_n$ maps to a scalar multiple $a_n\til{m}_n$ of $\til{m}_n$. As such, choosing $a_n=0$ whenever $n\in C^*$ ensures by multiplicativity that the coefficients of $\til{m}_\lam$ are zero as required in the cases where $\lam_i\in C^*$ for some $i$. Conversely, if $S^\lam(G)\neq\emptyset$, then there exist positive integers $k_1,\dots,k_{\ell(\lam)}$ such that $\lam_i\in C_{k_i}$ for all $i$. Once again, by Lemma \ref{lem:source_components}, $$(-1)^{|V(G)|-\ell(\lam)}[p_\lam]X_G=|\mc{AO}^\lam(G)|,$$ where $\mc{AO}^\lam(G)$ is the set of all acyclic orientations on $G$ where the source components have weights $\lam_i$. As argued in the proof of Proposition 1.6, we have $$|\mc{AO}^\lam(G)|=\sum_{\pi\in V^\lam(G)}\prod_{i=1}^{\ell(\lam)}|\mc{AO}(G|_{\pi_i})|,$$ since for each choice $\pi$ of a partition of $V(G)$ into source components with weights $\lam$, we can assign acyclic orientations onto each source component independently of each other, and once that is done, it uniquely determines the acyclic orientation of the whole graph $G$ since the directions of the edges between different source components must be directed from the later source component to the earlier one. But if $\lam_i\in C_{k_i}$, then the source component (i.e. the induced subgraph $G|_{\pi_i}$) with this weight is isomorphic to the clique $K_{k_i}$. Since this is true for all $\pi\in V^\lam(G)$ and $V^\lam(G)=S^\lam(G)$ in such a case, we get $$|\mc{AO}^\lam(G)|=|S^\lam(G)|\prod_{i=1}^{\ell(\lam)}|\mc{AO}(K_{k_i})|.$$ As argued in \S\ref{subsec:clique_maps_proof}, $|\mc{AO}(K_{k_i})|=(k_i-1)!$, and thus $$(-1)^{|V(G)|-\ell(\lam)}[p_\lam]X_G=|S^\lam(G)|\prod_{i=1}^{\ell(\lam)}(k_i-1)!=[\til{m}_\lam]X_{\ol{G}}\cdot\prod_{i=1}^{\ell(\lam)}(k_i-1)!.$$ So if $p_n$ maps to $(-1)^{k-1}\til{m}_n/(k-1)!$ whenever $n\in C_k$, then when $\lam_i\in C_{k_i}$ for each $i$ we have $$[p_\lam]X_G\cdot p_\lam=[p_\lam]X_G\cdot\prod_{i=1}^{\ell(\lam)}p_{\lam_i}\mapsto(-1)^{|V(G)|-\ell(\lam)}[\til{m}_\lam]X_{\ol{G}}\left(\prod_{i=1}^{\ell(\lam)}\frac{(-1)^{k_i-1}}{(k_i-1)!}\til{m}_{\lam_i}\right)\prod_{i=1}^{\ell(\lam)}(k_i-1)!=[\til{m}_\lam]X_{\ol{G}}\cdot\til{m}_\lam,$$ as required. Lastly, just like in Proposition \ref{prop:weighted_triangle_free_maps}, if there exist integer(s) $n\in\mathbb N\setminus(C^*\sqcup C_1\sqcup C_2\sqcup\cdots)$, then we can map these $p_n$'s to any multiple of $\til{m}_n$. This is because $[p_\lam]X_G=[\til{m}_\lam]X_{\ol{G}}=0$ whenever $\lam$ contains a part with value $n\in\mathbb N\setminus(C^*\sqcup C_1\sqcup C_2\sqcup\cdots)$, since $V^\lam(G)=\emptyset$ in such a case. \qed

\section{Examples of complements maps for specific graphs}\label{sec:examples}

\subsection{Existence for weighted triangle-containing non-clique families}
In contrast to the unweighted case as discussed in Proposition \ref{prop:c_equals_zero_or_half}, it is possible to define a Hopf algebra morphism $\phi:\Lam\rightarrow\til{\Lam}$ with the following properties: \begin{enumerate}[label={(\arabic*)}]
\item $\phi(X_G)=X_{\ol{G}}$ for all weighted graphs $G$ in a family $F$;
\item The family $F$ satisfies the conditions of Proposition \ref{prop:hopf_complement_maps}, i.e. $F=\{G_n\}_{n\in\mathbb N}$ where $G_n$ has weight $n$, each $G_n$ is connected, and all induced subgraphs of $G_n$ are isomorphic to a disjoint union of $G_i$'s from the same family;
\item $\phi(p_n)\neq0$ for some $n\in C$ (where the sets $V$, $E$, and $C$ are the same ones defined in Proposition \ref{prop:weighted_triangle_free_maps});
\item $F$ contains graphs that are not weighted cliques;
\item $F$ contains graphs with triangles;
\item In fact, we can make $F$ contain cliques of arbitrary size by choosing $B$ appropriately.
\end{enumerate} We describe a method to construct such a family below.

\begin{example}\label{example:binary_paths_cliques}
Write $\mathbb{N}_0 = \{0,1,2,\dots\}$ for the set of nonnegative integers. Fix a nonempty (possibly infinite) collection $B$ of disjoint finite subsets $B_k \se \{2^m\mid m\in\mathbb N_{0}\}=\{1,2,4,8,\dots\}$ such that each $B_k$ satisfies the following properties: \begin{itemize}
\item $|B_k|\geq3$;
\item $B_k$ consists of consecutive powers of 2.
\end{itemize}
For each $n\in\mathbb N$, let $\lam^{\tn{bin}(n)}$ denote the partition of $n$ whose parts $\lam^{\tn{bin}(n)}_i$ are the terms present in the binary expansion of $n$, arranged in descending order as usual (i.e. $\lam^{\tn{bin}(n)}_i>\lam^{\tn{bin}(n)}_j$ whenever $i<j$). For example, $\lam^{\tn{bin}(45)}=(32,8,4,1)$. Now let $G_n$ be the weight $n$ graph with $\ell(\lam^{\tn{bin}(n)})$ vertices whose vertex weights are given by the parts of $\lam^{\tn{bin}(n)}$. The vertices with weights $\lam^{\tn{bin}(n)}_i$ and $\lam^{\tn{bin}(n)}_j$ form an edge in $G_n$ if and only if one or both of the following holds: \begin{itemize}
\item $|i-j|=1$;
\item $\big\{\lam^{\tn{bin}(n)}_i,\lam^{\tn{bin}(n)}_j\big\}\subset B_k$ for some $B_k\in B$.
\end{itemize}
In other words, $G_n$ is obtained by first constructing the path graph connecting adjacent (when ordered by size) terms in the binary expansion of $n$, and then making consecutive terms into a clique if they come from the same $B_k$. For instance, if $B=\{\{1,2,4\},\{16,32,64,128,256\}\}$, then \vspace{-15pt}
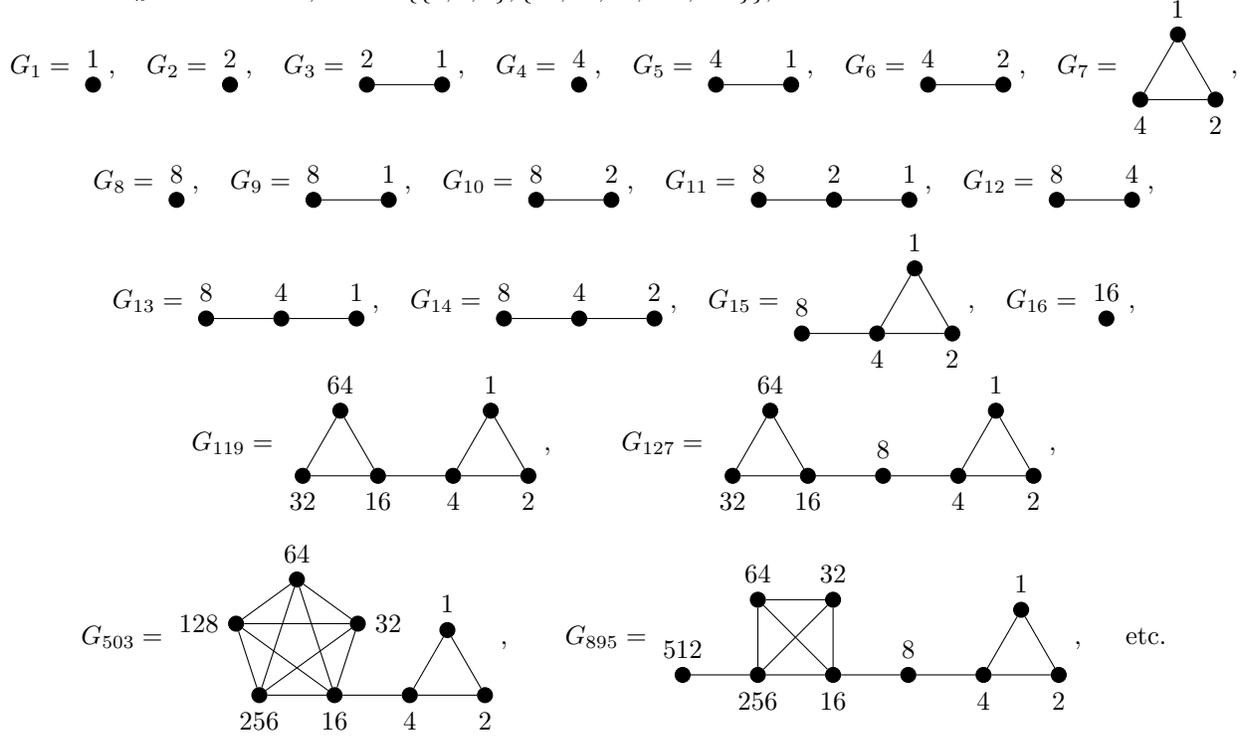
$$G_1=\vcenter{\hbox{\begin{tikzpicture}
\graph[nodes={draw, circle, fill=black, inner sep=2pt}, empty nodes, no placement]{a[x=0,y=0,label=above:1];};
\end{tikzpicture}}},\hspace{10pt}
G_2=\vcenter{\hbox{\begin{tikzpicture}
\graph[nodes={draw, circle, fill=black, inner sep=2pt}, empty nodes, no placement]{a[x=0,y=0,label=above:2];};
\end{tikzpicture}}},\hspace{10pt}
G_3=\vcenter{\hbox{\begin{tikzpicture}
\graph[nodes={draw, circle, fill=black, inner sep=2pt}, empty nodes, no placement]{a[x=0,y=0,label=above:2]--b[x=1,y=0,label=above:1];};
\end{tikzpicture}}},\hspace{10pt}
G_4=\vcenter{\hbox{\begin{tikzpicture}
\graph[nodes={draw, circle, fill=black, inner sep=2pt}, empty nodes, no placement]{a[x=0,y=0,label=above:4];};
\end{tikzpicture}}},\hspace{10pt}
G_5=\vcenter{\hbox{\begin{tikzpicture}
\graph[nodes={draw, circle, fill=black, inner sep=2pt}, empty nodes, no placement]{a[x=0,y=0,label=above:4]--b[x=1,y=0,label=above:1];};
\end{tikzpicture}}},\hspace{10pt}
G_6=\vcenter{\hbox{\begin{tikzpicture}
\graph[nodes={draw, circle, fill=black, inner sep=2pt}, empty nodes, no placement]{a[x=0,y=0,label=above:4]--b[x=1,y=0,label=above:2];};
\end{tikzpicture}}},\hspace{10pt}
G_7=\vcenter{\hbox{\begin{tikzpicture}
\graph[nodes={draw, circle, fill=black, inner sep=2pt}, empty nodes, no placement]{a[x=0,y=0,label=below:4]--b[x=1,y=0,label=below:2]--c[x=0.5,y=0.866,label=above:1]--a;};
\end{tikzpicture}}},\vspace{-5pt}$$
$$G_8=\vcenter{\hbox{\begin{tikzpicture}
\graph[nodes={draw, circle, fill=black, inner sep=2pt}, empty nodes, no placement]{a[x=0,y=0,label=above:8];};
\end{tikzpicture}}},\hspace{10pt}
G_9=\vcenter{\hbox{\begin{tikzpicture}
\graph[nodes={draw, circle, fill=black, inner sep=2pt}, empty nodes, no placement]{a[x=0,y=0,label=above:8]--b[x=1,y=0,label=above:1];};
\end{tikzpicture}}},\hspace{10pt}
G_{10}=\vcenter{\hbox{\begin{tikzpicture}
\graph[nodes={draw, circle, fill=black, inner sep=2pt}, empty nodes, no placement]{a[x=0,y=0,label=above:8]--b[x=1,y=0,label=above:2];};
\end{tikzpicture}}},\hspace{10pt}
G_{11}=\vcenter{\hbox{\begin{tikzpicture}
\graph[nodes={draw, circle, fill=black, inner sep=2pt}, empty nodes, no placement]{a[x=0,y=0,label=above:8]--b[x=1,y=0,label=above:2]--c[x=2,y=0,label=above:1];};
\end{tikzpicture}}},\hspace{10pt}
G_{12}=\vcenter{\hbox{\begin{tikzpicture}
\graph[nodes={draw, circle, fill=black, inner sep=2pt}, empty nodes, no placement]{a[x=0,y=0,label=above:8]--b[x=1,y=0,label=above:4];};
\end{tikzpicture}}},$$
$$G_{13}=\vcenter{\hbox{\begin{tikzpicture}
\graph[nodes={draw, circle, fill=black, inner sep=2pt}, empty nodes, no placement]{a[x=0,y=0,label=above:8]--b[x=1,y=0,label=above:4]--c[x=2,y=0,label=above:1];};
\end{tikzpicture}}},\hspace{10pt}
G_{14}=\vcenter{\hbox{\begin{tikzpicture}
\graph[nodes={draw, circle, fill=black, inner sep=2pt}, empty nodes, no placement]{a[x=0,y=0,label=above:8]--b[x=1,y=0,label=above:4]--c[x=2,y=0,label=above:2];};
\end{tikzpicture}}},\hspace{10pt}
G_{15}=\vcenter{\hbox{\begin{tikzpicture}
\graph[nodes={draw, circle, fill=black, inner sep=2pt}, empty nodes, no placement]{a[x=0,y=0,label=above:8]--b[x=1,y=0,label=below:4]--c[x=2,y=0,label=below:2]--d[x=1.5,y=0.866,label=above:1]--b;};
\end{tikzpicture}}},\hspace{10pt}
G_{16}=\vcenter{\hbox{\begin{tikzpicture}
\graph[nodes={draw, circle, fill=black, inner sep=2pt}, empty nodes, no placement]{a[x=0,y=0,label=above:16];};
\end{tikzpicture}}},\vspace{-10pt}$$
$$G_{119}=\vcenter{\hbox{\begin{tikzpicture}
\graph[nodes={draw, circle, fill=black, inner sep=2pt}, empty nodes, no placement]{a[x=1,y=0,label=below:16]--b[x=0.5,y=.866,label=above:64]--c[x=0,y=0,label=below:32]--a--d[x=2,y=0,label=below:4]--e[x=3,y=0,label=below:2]--f[x=2.5,y=0.866,label=above:1]--d;};
\end{tikzpicture}}},\hspace{25pt}
G_{127}=\vcenter{\hbox{\begin{tikzpicture}
\graph[nodes={draw, circle, fill=black, inner sep=2pt}, empty nodes, no placement]{a[x=1,y=0,label=below:16]--b[x=0.5,y=.866,label=above:64]--c[x=0,y=0,label=below:32]--a--d[x=2,y=0,label=above:8]--e[x=3,y=0,label=below:4]--f[x=4,y=0,label=below:2]--g[x=3.5,y=0.866,label=above:1]--e;};
\end{tikzpicture}}},$$
$$G_{503}=\vcenter{\hbox{\begin{tikzpicture}
\graph[nodes={draw, circle, fill=black, inner sep=2pt}, empty nodes, no placement]{e[x=1,y=0,label=below:16]--a[x=0,y=0,label=below:256]--b[x=-0.309,y=0.951,label=left:128]--c[x=0.5,y=1.539,label=above:64]--d[x=1.309,y=0.951,label=right:32]--e--f[x=2,y=0,label=below:4]--g[x=3,y=0,label=below:2]--h[x=2.5,y=0.866,label=above:1]--f;a--c--e--b--d--a;};
\end{tikzpicture}}},\hspace{20pt}
G_{895}=\vcenter{\hbox{\begin{tikzpicture}
\graph[nodes={draw, circle, fill=black, inner sep=2pt}, empty nodes, no placement]{a[x=0,y=0,label=above:512]--b[x=1,y=0,label=below:256]--c[x=1,y=1,label=above:64]--d[x=2,y=1,label=above:32]--e[x=2,y=0,label=below:16]--f[x=3,y=0,label=above:8]--g[x=4,y=0,label=below:4]--h[x=5,y=0,label=below:2]--i[x=4.5,y=0.866,label=above:1]--g;b--d;c--e;b--e;};
\end{tikzpicture}}},\hspace{15pt}\tn{etc.}$$
\end{example}

\begin{prop} The family $F=\{G_n\}_{n\in\mathbb N}$ described in Example \ref{example:binary_paths_cliques} satisfies properties (1) to (5) described at the beginning of this section. \end{prop}
\begin{proof}
Using the definitions of the $C_m$'s from Proposition \ref{prop:the_culmination}, we see by the definition of $G_n$ that $$C_m=\begin{cases}\{2^k\mid k\in\mathbb N_0\}&\text{if }m=1,\\\{2^k+2^l\mid k,l\in\mathbb N_0,\;k\neq l\}&\text{if } m=2,\\\displaystyle\bigsqcup_{k=1}^{|B|}\bigsqcup_{\substack{S\se B_k\\|S|=m}}\left\{\sum_{x\in S}x\right\}&\text{if }m\geq3.\end{cases}$$ In other words, $V=C_1$ consists of all the powers of 2, $E=C_2$ consists of all the positive integers whose binary expansion contains exactly 2 terms, and for $m\geq3$, $C=\mathbb N\setminus(C_1\sqcup C_2)$ consists of all positive integers whose binary expansion contains at least 3 terms, with $C_m$ ($m\geq3$) being the subset of $C$ whose binary expansion contains exactly $m$ terms, all of which come from the same $B_k$. These $C_m$'s are disjoint due to the uniqueness of binary expansions.

If the binary expansion of $n$ contains at least 3 terms, and they are not all from the same $B_k$, then in particular, the largest and smallest terms (i.e. $\lam^{\tn{bin}(n)}_1$ and $\lam^{\tn{bin}(n)}_\ell$, where $\ell:=\ell(\lam^{\tn{bin}(n)})$ is the number of terms) cannot be from the same $B_k$, for if they were, all the terms in between would be from the same $B_k$ since each $B_k$ consists of consecutive powers of 2. Therefore, since $\ell\geq3$ implies $|\ell-1|>1$, the vertices with weights $\lam^{\tn{bin}(n)}_1$ and $\lam^{\tn{bin}(n)}_\ell$ do not form an edge. As such, $G_n$ cannot be a clique in such a case. This means that $C^*=\mathbb N\setminus\big(\bigsqcup_{m=1}^\infty C_m\big)$, and in particular, this is also disjoint from the $C_m$'s. As such, once we know property (2), we can apply Proposition \ref{prop:the_culmination} to immediately deduce property (1), where $\phi$ is the map defined in \S\ref{subsec:culmination_proof}. Property (3) then follows because $C = C^* \sqcup C_3\sqcup C_4\sqcup \dots$ contains elements from at least one $C_m$ with $m\ge 3$ (as $B$ is nonempty), and the map from the proof of Proposition \ref{prop:the_culmination} in \S\ref{subsec:culmination_proof} then sends the corresponding $p_n$ terms to nonzero multiples of their associated $\til{m}_n$'s. Property (4) also follows immediately, since for $n=b_1+b'_1+d$ where $b_1,b'_1\in B_1$ and $d\notin B_1$, $G_n$ is a 3-vertex path, which in particular is connected but not a clique. Similarly, for $n=b_1+b'_1+b''_1$ where $b_1,b'_1,b''_1\in B_1$, $G_n$ is a triangle, and so property (5) follows. We can make $F$ contain a clique of arbitrary size $m$ by requiring $B$ to contain an element $B_k$ of cardinality $m$, and hence property (6) follows. It therefore remains to prove property (2).

Consider an induced subgraph $G_n|_S$ of $G_n$. Suppose $S$ is chosen such that $G_n|_S$ is connected. Let $B_k^{(n)}:=B_k\cap\lam^{\tn{bin}(n)}$ and $B^{((n))}_k:=B^{(n)}_k\setminus\{\max B^{(n)}_k, \min B^{(n)}_k\}$. Notice that every edge whose endpoints are not both in the same $B_k$ is a bridge in $G_n$, i.e. an edge whose removal would disconnect $G_n$. All such edges connect the vertices with weights $\lam^{\tn{bin}(n)}_i$ and $\lam^{\tn{bin}(n)}_{i+1}$ for some $i$, and as such, if $S$ contains the vertices with weights $\lam^{\tn{bin}(n)}_i$ and $\lam^{\tn{bin}(n)}_j$, then for $G_n|_S$ to be connected, $S$ must also contain the vertices with weights $\lam^{\tn{bin}(n)}_h$ for every $h$ between $i$ and $j$, except possibly for those where $\lam^{\tn{bin}(n)}_h\in B^{((n))}_k$. Beyond that, any choice of $S\cap B^{((n))}_k$ is possible, and in such a case, the corresponding vertices also form a clique in $G_n|_S$. All other edges in $G_n|_S$ still connect adjacent vertices when arranged in order of weight, and hence $G_n|_S$ belongs to the family $F$, as required.
\end{proof}

\begin{remark}
We can generalize this example a bit further by dropping the requirement that the $B_k$'s be disjoint. The proof proceeds in the same fashion since induced subgraphs of a union (not necessarily disjoint) of cliques are still unions of cliques on the appropriate vertex subsets. This also allows us to redefine the $G_n$'s as unit interval graphs by requiring the intervals corresponding to the vertices in each $B_k$ to mutually overlap, and for adjacent terms in the binary expansion to overlap as well.
\end{remark}

\subsection{Maps (or lack thereof) for single graphs}
So far, we have explored the possible graded Hopf algebra maps that simultaneously send the CSFs of a whole family of graphs to those of their complements. But what if we only need to do so for a single graph rather than an entire family? As we will see, this offers some, but not too many, extra possibilities. In particular, it is still impossible to define such a map most of the time.
\begin{example} For the following graph $G$, we can find a graded Hopf algebra morphism $\phi:\Lam\rightarrow\til{\Lam}$ sending $X_G$ to $X_{\ol{G}}$ for which there exists a value of $n$  such that $a_n\notin\{0\}\cup\left\{\frac{(-1)^{k-1}}{(k-1)!}\right\}_{k\in\mathbb N}$ where $\phi(p_n)=a_n\til{m}_n$.
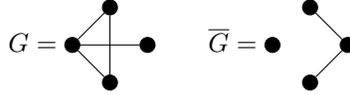
$$G=\vcenter{\hbox{\begin{tikzpicture}
\graph[nodes={draw, circle, fill=black, inner sep=2pt}, empty nodes, no placement]{a[x=0,y=0]--b[x=0.5,y=0.5]--c[x=0.5,y=-0.5]--a--d[x=1,y=0];};
\end{tikzpicture}}}
\hspace{20pt}\ol{G}=\vcenter{\hbox{\begin{tikzpicture}
\graph[nodes={draw, circle, fill=black, inner sep=2pt}, empty nodes, no placement]{a[x=0,y=0];b[x=0.5,y=0.5]--d[x=1,y=0]--c[x=0.5,y=-0.5];};
\end{tikzpicture}}}$$
\end{example} \begin{proof}
We compute: 
\begin{align*}
    X_G&=p_{1111}-4p_{211}+ p_{22}+4p_{31}-2p_4, \\
    X_{\ol{G}}&=\til{m}_{1111}+4\til{m}_{211}+ \til{m}_{22}+\til{m}_{31}.
\end{align*} 
It follows that $$\phi(X_G)=a_1^4\til{m}_{1111}-4a_1^2a_2\til{m}_{211}+a_2^2\til{m}_{22}+4a_1a_3\til{m}_{31}-2a_4\til{m}_4,$$ so equating coefficients with that of $X_{\ol{G}}$, we see that $(a_1^4,-4a_1^2a_2,a_2^2,4a_1a_3,a_4)=(1,4,1,1,0)$. This is satisfied (though not uniquely) by $(a_1,a_2,a_3,a_4)=(1,-1,\frac14,0)$. Note in particular that $a_3$'s denominator is not the factorial of a whole number. \
\end{proof}
Notice that to apply Proposition \ref{prop:clique_maps} or \ref{prop:unweighted_triangle_free_maps}, we needed each graph in the family to be a clique or a triangle-free graph, respectively, neither condition of which is satisfied by $G$ in the example above. This confirms our earlier assertion of the existence of extra possibilities when we restrict to single graphs. However, this is still a ``nice'' example in the sense that all the $a_n$'s are rational, something that need not always be the case as we will see in the more exotic example below:
\begin{example}
For the following weighted graph $G$, we can find a graded Hopf algebra morphism $\phi:\Lam\rightarrow\til{\Lam}$ sending $X_G$ to $X_{\ol{G}}$ and $p_n$ to $a_n\til{m}_n$, but only if we allow non-real coefficients $a_n$.
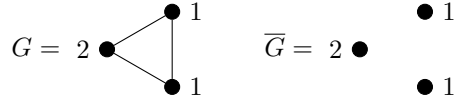
$$G=\vcenter{\hbox{\begin{tikzpicture}
\graph[nodes={draw, circle, fill=black, inner sep=2pt}, empty nodes, no placement]{a[x=-0.866,y=0,label=left:2]--b[x=0,y=0.5,label=right:1]--c[x=0,y=-0.5,label=right:1]--a};
\end{tikzpicture}}}
\hspace{20pt}\ol{G}=\vcenter{\hbox{\begin{tikzpicture}
\graph[nodes={draw, circle, fill=black, inner sep=2pt}, empty nodes, no placement]{a[x=-0.866,y=0,label=left:2];b[x=0,y=0.5,label=right:1];c[x=0,y=-0.5,label=right:1];};
\end{tikzpicture}}}$$
\end{example}
\begin{proof}
Here we have  
\begin{align*}
    X_G&=p_{211}-p_{22}-2p_{31}+2p_4, \\
    X_{\ol{G}}&=\til{m}_{211}+\til{m}_{22}+ 2\til{m}_{31}+\til{m}_{4}.
\end{align*}
This gives the equations $(a_1^2a_2,-a_2^2,-2a_3a_1,2a_4)=(1,1,2,1)$. Straightforward computation shows that the solutions are of the form $(a_1,a_2,a_3,a_4)=(\zeta_8,\zeta_8^6,\zeta_8^3,\frac12)$ where $\zeta_8$ is a primitive $8^\tn{th}$ root of unity. In particular, all solutions involve elements of $\mathbb C\setminus\mathbb R$.
\end{proof}
Finally, $\phi$ need not exist even for single graphs. This actually happens most of the time, especially for large graphs, since there may be up to $p(n)$ terms in the $p$-basis expansion of $X_G$ if $G$ has total weight $n$, where $p(n)$ is the number of partitions of $n$. Requiring each such $p_\lam$ to map to the appropriate multiple (i.e. the one that makes the coefficients of $\phi(X_G)$ and $X_{\ol{G}}$ equal) of $\til{m}_\lam$ would therefore induce a system of up to $p(n)$ equations in the $n$ unknowns $a_1,\dots,a_n$. Since $p(n)\gg n$ for large $n$, this is an overdetermined system. We illustrate an example of such a case below.
\begin{example}
For the following graph $G$, no graded Hopf algebra morphism $\phi:\Lam\rightarrow\til{\Lam}$ sends $X_G$ to $X_{\ol{G}}$.
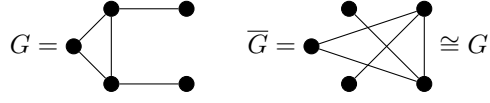
$$G=\vcenter{\hbox{\begin{tikzpicture}
\graph[nodes={draw, circle, fill=black, inner sep=2pt}, empty nodes, no placement]{a[x=0,y=0]--b[x=0.5,y=0.5]--c[x=0.5,y=-0.5]--a;b--d[x=1.5,y=0.5];c--e[x=1.5,y=-0.5];};
\end{tikzpicture}}}
\hspace{20pt}\ol{G}=\vcenter{\hbox{\begin{tikzpicture}
\graph[nodes={draw, circle, fill=black, inner sep=2pt}, empty nodes, no placement]{a[x=0,y=0]--d[x=1.5,y=0.5]--e[x=1.5,y=-0.5]--a;b[x=0.5,y=0.5]--e;c[x=0.5,y=-0.5]--d;};
\end{tikzpicture}}}\cong G$$
\end{example}
\begin{proof}
Suppose for contradiction that such a $\phi$ exists, with $\phi(p_n)=a_n\til{m}_n$. Now observe:
\begin{alignat*}{10}X_G&\,=\,&p_{11111}\,&\,-\,&\,5p_{2111}\,&\,+\,&\,3p_{221}\,&\,+\,&\,6p_{311}&-2p_{32}-5p_{41}+2p_{5},\\X_{\ol{G}}&\,=\,&\til{m}_{11111}&\,+\,&5\til{m}_{2111}&\,+\,&3\til{m}_{221}&\,+\,&\til{m}_{311}&.\end{alignat*}
The first three coefficients imply that $|a_1|=|a_2|=1$, while the next one implies that $|a_3|=\frac16$. In particular, we must have $a_2\neq0$ and $a_3\neq0$. However, this would mean that $\phi(-2p_{32})=-2a_2a_3\til{m}_{32}\neq0$, which is a contradiction since $[\til{m}_{32}]X_{\ol{G}}=0$.
\end{proof}

\section{\texorpdfstring{$K$}{K}-analogues}\label{sec:K-analogues}
\counterwithin{theorem}{section}
\setcounter{theorem}{0}

In this section, we prove $K$-analogues of several of our results from earlier sections. Crew, Pechenik, and Spirkl \cite{crew2023kromatic} define a \emph{\tb{\tcb{proper set coloring}}} as an assignment of a nonempty set $\kappa(v)$ of positive integer valued colors to each vertex $v$ of $G$ such that adjacent vertices get nonoverlapping color sets, and they define the \emph{\tb{\tcb{Kromatic symmetric function (KSF)}}} as $$\ol{X}_G := \sum_\kappa \prod_{v\in V(G)}\prod_{i\in \kappa(v)} x_i^{w(v)},$$ where $\kappa$ ranges over all proper set colorings of $G$. Thus, each monomial corresponds to a proper set coloring, and within the monomial, the exponent on the variable $x_i$ is the weight of the set of vertices getting color $i$. Note that the lowest degree terms of $\ol{X}_G$ are exactly the terms of the CSF $X_G$. Since $\Lam$ only allows power series of bounded degree, $\ol{X}_G$ technically lives not in $\Lam$ but in its \emph{\tb{\tcb{completion}}} $\ol{\Lam}$, which allows symmetric power series of unbounded degree. 

The KSF was intended to be a $K$-analogue of the CSF in the sense of $K$-theory, although it is unknown whether it actually has a $K$-theoretic interpretation.  However, the KSF does have a Hopf algebra interpretation due to Marberg \cite{marberg2023kromatic}. He defines a modified Hopf algebra structure $\mWGraphs$ by first taking the completion of $\WGraphs$ (so allowing linear combinations of infinitely many different graphs) and then also changing the coproduct to be $$\blacktriangle G := \sum_{S\cup T = V(G)} G|_S\otimes G|_T,$$ the difference from the usual coproduct being that the vertex sets $S$ and $T$ are allowed to overlap.

For these Hopf algebras that allow infinite linear combinations of basis vectors, the usual characters do not make sense (sending all $p_\lam$'s to 1 or all edgeless graphs to 1) because they send elements that are infinite combinations of basis vectors to $\infty$. To fix this, Marberg instead defines modified characters that are maps to $\K[[t]]$ instead of to $\K$: $$\ol{\zeta}_{\Lam}(p_\lam) := t^{|\lam|},\hspace{1cm}\ol{\zeta}_{\WGraphs} := 0^{|E(G)|}\prod_{v\in V(G)}t^{w(v)}.$$ He refers to these modified Hopf algebra structures together their modified characters as \emph{\tb{\tcb{linearly compact (LC)-Hopf algebras}}} He then shows that there is a unique LC-Hopf algebra map $(\mWGraphs, \ol{\zeta}_{\WGraphs}) \to (\ol{\Lam},\ol{\zeta}_\Lam),$ and that it is given by $$G \mapsto \ol{X}_G.$$

$\ol{\Lam}$ has the same coproduct as $\Lam$, so we can similarly make the completion $\ol{\til{\Lam}}$ of Tusjie's $\til{\Lam}$ ring into a Hopf algebra by using the same coproduct. Now we can define \emph{\tb{\tcb{Kromatic pseudobases}}} of $\ol{\Lam}$ in the same manner as the chromatic bases:

\begin{prop}\label{prop:kromatic}
    If $G_1,G_2,\dots$ is a sequence of connected graphs such that $G_n$ has total weight $n$, then the set of KSFs $\ol{X}_{G_\lam}$ forms a multiplicative basis for $\Lam$ (which we will call a {\tb{\tcb{Kromatic pseudobasis}}}).
\end{prop}

Following Marberg \cite{marberg2023kromatic}, we use the word \emph{\tb{\tcb{pseudobasis}}} instead of basis to indicate that each element of $\ol{\Lam}$ can be written as a unique but possibly \emph{infinite} linear combination of the basis vectors, rather than necessarily as a finite linear combination of them.

\begin{proof}
    The multiplicative part is automatic, since the KSF multiplies over disjoint unions just like the CSF, i.e. $\ol{X}_{G\sqcup H} = \ol{X}_G \ol{X}_H.$ To show that it is actually a pseudobasis, assume for contradiction we had some (possibly infinite) nontrivial relation $a_1 \ol{X}_{G_{\lam^1}} + a_2 \ol{X}_{G_{\lam^2}}+\dots$ for some partitions $\lam^i$. Now note that the lowest degree terms in $\ol{X}_{G_{\lam^i}}$ have degree $|\lam^i|$ and form the CSF $X_{G_{\lam^i}}.$ Thus, if we let $d$ be the lowest degree among all terms showing up among any of the $\ol{X}_{G_{\lam^i}}$'s, the degree $d$ terms are precisely the CSFs $X_{G_{\lam^i}}$ for the partitions $\lam^i$ in our set that satisfy $|\lam^i|=d$, and there can only be finitely many such CSFs because there are only finitely many partitions $\lam$ with $|\lam|=d$. But then we would get a nontrivial linear relation among finitely many of the CSFs $X_{G_\lam}$, which cannot happen because the CSFs $X_{G_\lam}$ are known from \cite{aliniaeifard2021extended} to form a chromatic basis. Thus, the $\ol{X}_{G_\lam}$'s form a pseudobasis for $\ol{\Lam}$.
\end{proof}

We know from \cite{pierson2025graphs} that in $\ol{\til{\Lam}}$, the KSF multiplies over joins, just like the CSF, i.e. $\ol{X}_{G\odot H} = \ol{X}_G \odot \ol{X}_H.$ Thus, for a sequence of graphs $G_1,G_2,\dots$, the set of KSFs $\ol{X}_{\ol{G_\lam}}=\ol{X}_{\ol{G_{\lam_1}}\odot \dots \odot \ol{G_{\lam_\ell}}}$ will form a multiplicative \emph{\tb{\tcb{coKromatic pseudobasis}}} for $\ol{\til{\Lam}}$ as long as there are no linear relations among them. The possible coKromatic pseudobases look essentially the same as our possible cochromatic bases from Proposition \ref{prop:cochromatic}:

\begin{prop}\label{prop:cokromatic}
    For a sequence of graphs $G_1,G_2,\dots,$ the KSFs $\ol{X}_{\ol{G_\lam}}=\ol{X}_{\ol{G_{\lam_1}}\odot \dots \odot \ol{G_{\lam_\ell}}}$ form a pseudobasis for $\ol{\til{\Lam}}$ if and only if $G_n$ is a clique of total weight $n$ for every $n$.
\end{prop}

\begin{proof}
    To prove this, we will introduce some more of the terminology from \cite{crew2023kromatic}. The \emph{\tb{\tcb{$K$-theoretic augmented monomial symmetric functions}}} are $$\ol{\til{m}}_\lam:= \ol{X}_{K^\lam},$$ where $K^\lam$ is the weighted clique whose vertex weights are the parts $\lam_1,\dots,\lam_\ell$ of $\lam$. The authors of \cite{crew2023kromatic} show that the $\ol{\til{m}}_\lam$'s form a pseudobasis for $\ol{\Lam}$. It follows that they form a \emph{multiplicative} pseudobasis for $\ol{\til{\Lam}}$, since they satisfy $$\ol{\til{m}}_{\lam \sqcup \mu} = \ol{X}_{K^\lam \odot K^\mu} = \ol{X}_{K^\lam}\odot \ol{X}_{K^\mu} = \ol{\til{m}}_\lam \odot \ol{\til{m}}_\mu,$$ as KSFs multiply over joins in $\ol{\til{\Lam}}.$ The authors of \cite{crew2023kromatic} show that KSFs expand in the $\ol{\til{m}}_\lam$ pseudobasis as $$\ol{X}_G = \sum_{C\in{\textsf{SSC}}(G)} \ol{\til{m}}_{\lam(C)},$$ where ${\textsf{SSC}}(G)$ is the set of \emph{\tb{\tcb{stable set covers (SSCs)}}} of $G$, i.e. collections $C$ of stable sets in $G$ such that each vertex is in at least one of the stable sets, and $\lam(C)$ is the partition whose parts are the weights of the stable sets in $C$.

    Now for the $X_{\ol{G_\lam}}$'s to form a pseudobasis, we must be able to get all the $\ol{\til{m}}_n$'s as (possibly infinite) linear combinations of them. But we can only get an $\ol{\til{m}}_n$ by itself in the SSC expansion if some SSC consists of a single stable set of weight $n$. Thus, one of the graphs $\ol{G_{i_n}}$ needs to be edgeless with total weight $n$. Then by the same argument as in Proposition \ref{prop:cochromatic}, every $\ol{\til{m}}_n$, and hence every $\ol{\til{m}}_\lam,$ can be expressed as a polynomial in the $\ol{G_{i_n}}$'s, so we cannot add any additional graphs $\ol{G_n}$ with $n\ne i_j$ for some $j$, or the resulting set of KSFs would not be algebraically independent. Thus, the only possibility is that $\ol{G_n}$ is an edgeless graph of weight $n$ for each $n$, so $G_n$ is a clique of weight $n.$
\end{proof}

The $K$-analogue of Proposition \ref{prop:WGraphs_to_Lam_til} also holds by the same reasoning as for Proposition \ref{prop:WGraphs_to_Lam_til}:

\begin{prop}\label{prop:K_graph_maps}
    The map $G\mapsto \ol{X}_{\ol{G}}$ is a morphism of Hopf algebras from $\mWGraphs$ to $\ol{\til{\Lam}}$.
\end{prop}

\begin{proof}
    The product is preserved because $\ol{X}_{\ol{G\sqcup H}}=\ol{X}_{\ol{G}\odot\ol{H}} = \ol{X}_{\ol{G}}\odot \ol{X}_{\ol{H}}.$ For the coproduct, the fact that $G \mapsto \ol{X}_G$ is a Hopf algebra map implies that $$\Delta \ol{X}_G = \sum_{S\cup T = V(G)} \ol{X}_{G|_S}\otimes \ol{X}_{G|_T},$$ where $\Delta$ denotes the coproduct on $\ol{\Lam}$ and hence also using the coproduct on $\ol{\til{\Lam}}$, since we defined the coproducts to be the same. Because of the fact that $\ol{G}|_S = \ol{G|_S}$, we also have $$\Delta \ol{X}_{\ol{G}} = \sum_{S\cup T = V(G)} \ol{X}_{\ol{G}|_S}\otimes \ol{X}_{\ol{G}|_T} = \sum_{S\cup T = V(G)} \ol{X}_{\ol{G|_S}}\otimes \ol{X}_{\ol{G|_T}}.$$ These terms are precisely the images under our $G\mapsto \ol{X}_{\ol{G}}$ map of the corresponding terms in the coproduct $$\blacktriangle G := \sum_{S\cup T = V(G)} G|_S\otimes G|_T,$$ so the coproduct is preserved.
\end{proof}

We can also prove the $K$-analogue of Proposition \ref{prop:hopf_complement_maps} using the same reasoning as for Proposition \ref{prop:hopf_complement_maps}:

\begin{prop}\label{prop:K_complement_maps}
    If $G_1,G_2,\dots$ are connected graphs such that $G_n$ has weight $n$ and all its induced subgraphs are disjoint unions of $G_i$'s, then $\ol{X}_{G_\lam} \mapsto \ol{X}_{\ol{G_\lam}}$ is a Hopf algebra map from $\ol{\Lam}$ to $\ol{\til{\Lam}}.$
\end{prop}

\begin{proof}
    The product is preserved because $$\ol{X}_{G_\lam}\ol{X}_{G_\mu} = \ol{X}_{G_\lam \sqcup G_\mu} \mapsto \ol{X}_{\ol{G_\lam \sqcup G_\mu}} = \ol{X}_{\ol{G_\lam}\odot \ol{G_\mu}} = \ol{X}_{\ol{G_\lam}}\odot \ol{X}_{\ol{G_\mu}}.$$ The coproduct is also preserved because we can write $$\Delta \ol{X}_{G_\lam} = \sum_{\mu,\nu} a_{\mu\nu}^\lam(\ol{X}_{G_\mu}\otimes \ol{X}_{G_\nu}),$$ where $a_{\mu\nu}^\lam$ counts the number of pairs $S,T \se V(G)$ with $S\cup T = V(G)$ such that $(G_\lam)|_S = G_\mu$ and $(G_\lam)|_T = G_\nu.$ The pair $S,T$ then also satisfies $(\ol{G_\lam})|_S = \ol{G_\mu}$ and $(\ol{G_\lam})|_T = \ol{G_\nu}.$ Thus, we can also write $$\Delta \ol{X}_{\ol{G_\lam}} = \sum_{\mu,\nu} a_{\mu\nu}^\lam(\ol{X}_{\ol{G_\mu}}\otimes \ol{X}_{\ol{G_\nu}})$$ with the same structure constants $a_{\mu\nu}^\lam$, so the coproduct is preserved.
\end{proof}

Finally, we prove the $K$-analogue of \ref{prop:unweighted_triangle_free_maps}:

\begin{prop}\label{prop:K_triangle_free_map}
    There is a single Hopf algebra morphism $\ol{\Lam}\to\ol{\til{\Lam}}$ taking the KSFs of all unweighted triangle-free graphs to the KSFs of their complements.
\end{prop}

    The proof will use similar ideas to the proof of Proposition \ref{prop:unweighted_triangle_free_maps}, but we will need to introduce some more terminology. We defined \emph{\tb{\tcb{heaps}}} and \emph{\tb{\tcb{pyramids}}} in \S \ref{sec:p-expansion_lemma}. Given a fixed ordering on $V(G)$, a \emph{\tb{\tcb{Lyndon heap}}} (introduced by Lalonde \cite{lalonde1995lyndon} as an analogue of \emph{\tb{\tcb{Lyndon words}}}) is an aperiodic pyramid whose standard word is lexicographically minimal among the standard words of its rotations. To define what that means, we can assign to each heap an equivalence class of words where the alphabet is the set of vertices, such that the number of times each letter appears in the word equals the number of times the vertex appears in the heap, and such that $u$ must come before $v$ in the word whenever $uv$ is a directed edge in the heap. A heap is \emph{\tb{\tcb{aperiodic}}} if no word in its equivalence class is periodic. The \emph{\tb{\tcb{standard word}}} associated to a heap is the lexicographically maximal word in its equivalence class. The \emph{\tb{\tcb{rotation}}} operation on pyramids can be used to move any one of its vertices $v$ to the base to get a new pyramid. The process is to put $v$ and all vertices reachable from it at the end of the word, then move that portion of the word to the start, them repeat until the result is a pyramid with $v$ at the base (which Lalonde shows always happens eventually).

    The power expansion we will use here is in terms of the $\omega(\ol{p}'_\lam)$-basis, given by $$\omega(\ol{p}'_n) := \prod_{i\ge 1}(1-(-x_i)^n)-1, \hspace{1cm} \omega(\ol{p}'_\lam) := \omega(\ol{p}'_{\lam_1})\dots \omega(\ol{p}'_{\lam_\ell}),$$ which can alternatively be defined by applying the involution $\omega$ to the $\ol{p}'_\lam$ basis, defined by $$\ol{p}'_n := \prod_{i\ge 1}\frac1{1-x_i^n}-1,\hspace{1cm} \ol{p}'_\lam := \ol{p}'_{\lam_1}\dots\ol{p}'_{\lam_\ell}.$$ The power sum expansion we will use is:

    \begin{lemma}[\cite{pierson2025power}, Theorem 1.5]
        $[\omega(\ol{p}_n')]\ol{X}_G$ counts the number of ways to cover $G$ with distinct Lyndon heaps whose sizes are the parts of $\lam$, such that each vertex is used by at least one Lyndon heap.
    \end{lemma}

\begin{proof}[Proof of Proposition \ref{prop:K_triangle_free_map}]
    We claim that the desired map is $$\omega(\ol{p}_n') \mapsto \begin{cases}
        \ol{\til{m}}_n & n=1,2 \\
        0 & n\ge 3.
    \end{cases}$$
    Let $G$ be an unweighted triangle-free free graph, so we want to show that this map sends $\ol{X}_{G}$ to $\ol{X}_{\ol{G}}$, which we will prove by showing that it sends the $\omega(\ol{p}'_\lam)$ expansion of $\ol{X}_G$ to the $\ol{\til{m}}_\lam$ expansion of $\ol{X}_{\ol{G}}.$ 
    
    We know that $[\ol{\til{m}}_\lam]\ol{X}_{\ol{G}}$ counts the ways to cover $G$ with stable sets whose sizes are the parts of $\lam$. Since $G$ is triangle-free, $\ol{G}$ has no stable sets of size 3 or more, so the only $\ol{\til{m}}_\lam$ terms that show up in the SSC expansion of $\ol{X}_{\ol{G}}$ are ones where all parts of $\lam$ are 1's or 2's, and those terms count all ways to cover $\ol{G}$ with singleton vertices together with pairs of vertices that correspond to stable sets in $\ol{G}$, hence to edges in $G$ (where unlike in the CSF case, the stable sets are allowed to overlap).

    In the $\omega(\ol{p}'_\lam)$ expansion of $\ol{X}_G$, we only need to consider the $\omega(\ol{p}'_\lam)$ terms where all parts of $\lam$ are 1's or 2's, because all other terms will get sent to 0 by our map. Parts of $\lam$ of size 1 correspond to Lyndon heaps on 1 vertex, and there is always a unique way to make a single vertex into a Lyndon heap, so these are essentially just singleton vertices. Parts of $\lam$ of size 2 correspond to Lyndon heaps on 2 vertices. Such a heap cannot use two disconnected vertices since pyramids must be connected, and it cannot use the same vertex twice since then it would be periodic. Thus, it must come from an edge, and the only choice is to direct the first vertex of the edge (under the ordering on $V(G)$) towards the second vertex, since the other option would be the rotation of that one and would have a larger standard word, so it would not be the Lyndon heap in its equivalence class of rotations. Thus, Lyndon heaps of size 2 just correspond to edges. So, covering $G$ with Lyndon heaps of sizes 1 and 2 is the same thing as covering it with vertices and edges, so $[\omega(\ol{p}'_\lam)]\ol{X}_G = [\ol{\til{m}}_\lam]\ol{X}_{\ol{G}}$. Thus, our map sends $\ol{X}_G$ to $\ol{X}_{\ol{G}}$. 
    
    To see that this is a Hopf algebra morphism, we can apply Proposition \ref{prop:K_complement_maps} where $G_n$ is the $n$-vertex \emph{\tb{\tcb{star graph}}} (i.e. the graph with one central vertex connected to all other vertices), since all induced subgraphs of stars are disjoint unions of stars. (We could alternatively apply Proposition \ref{prop:K_complement_maps} and let $G_n$ be the $n$-vertex path, since all induced subgraphs of paths are disjoint unions of paths.)
\end{proof}

\section{Future directions}\label{sec:future}

One question is whether we can characterize more fully the families of weighted graphs from Proposition \ref{prop:hopf_complement_maps} with the property that all their connected induced subgraphs come from the same family. Relatedly, we can ask if more can be said about the possible graphs from Proposition \ref{prop:the_culmination} for given sets $C_1,C_2,\dots,C^*$. Another question is whether there are other interesting examples like the ones from \S\ref{sec:examples}, where there does or does not exist a Hopf algebra map sending the CSFs of specific graphs to the CSFs of their complements, or if there is a general characterization for those sorts of examples.

So far, the cochromatic bases are essentially the only interesting algebraically independent generating sets we know about for $\til{\Lam}$, so another question is whether there are other generating sets for $\til{\Lam}$ that are of combinatorial interest. Relatedly, all our $\Lam\to\til{\Lam}$ maps are either isomorphisms or non-surjective maps from $\Lam$ to $\til{\Lam}$, so one could ask if there are any interesting non-surjective maps going in the other direction, from $\til{\Lam}$ to $\Lam.$ One could also ask whether any of our maps send any of the other classic symmetric functions (such as Schur functions $s_\lam$ or the homogeneous symmetric functions $h_\lam$) to symmetric functions with an interesting combinatorial description.

Another direction would be to further investigate the Kromatic and coKromatic pseudobases from \S\ref{sec:K-analogues}. However, in that case the power sum coefficients count connected subsets where the same vertex can be used multiple times (Lyndon heaps), while the monomial coefficients count stable subsets where each vertex can only be used once, so it seems like it may be less easy in that context to map power sums to monomials in a way that simultaneously sends a large number of KSFs to their complements. The KSF maps also do not necessarily preserve the grading like the CSF maps do, which may make them more complicated to describe and analyze.

\section*{Acknowledgments}

We thank Jos\'{e} Aliste-Prieto, Logan Crew, Eric Marberg, Oliver Pechenik, and Sophie Spirkl for helpful discussions about this project.

\printbibliography

@article{crew2021complete,
  title={A complete multipartite basis for the chromatic symmetric function},
  author={Crew, Logan and Spirkl, Sophie},
  journal={SIAM Journal on Discrete Mathematics},
  volume={35},
  number={4},
  pages={2647--2661},
  year={2021},
  publisher={SIAM},
  url={https://doi.org/10.1137/20M1380314}
}

@article{tsujie2018chromatic,
  title={The chromatic symmetric functions of trivially perfect graphs and cographs},
  author={Tsujie, Shuhei},
  journal={Graphs and Combinatorics},
  volume={34},
  number={5},
  pages={1037--1048},
  year={2018},
  publisher={Springer},
  url={https://link.springer.com/article/10.1007/s00373-018-1928-2}
}

@article{marberg2023kromatic,
  title={Kromatic quasisymmetric functions},
  author={Marberg, Eric},
  journal={The Electronic Journal of Combinatorics},
  volume={32},
  issue={1},
  number={P1.11},
  year={2025},
  url={https://doi.org/10.37236/13207}
}

@article{cho2015chromatic,
  title={Chromatic bases for symmetric functions},
  author={Cho, Soojin and van Willigenburg, Stephanie},
  journal={The Electronic Journal of Combinatorics},
  volume={23},
  issue={1},
  number={P1.15},
  year={2016},
  url={https://doi.org/10.37236/5540}
}

@unpublished{crew2023kromatic,
  title={The Kromatic Symmetric Function: A $ K $-theoretic Analogue of $ X_G$},
  author={Crew, Logan and Pechenik, Oliver and Spirkl, Sophie},
  year={2023},
  url={https://arxiv.org/abs/2301.02177}
}

@article{penaguiao2020kernel,
  title={The kernel of chromatic quasisymmetric functions on graphs and hypergraphic polytopes},
  author={Penaguiao, Raul},
  journal={Journal of Combinatorial Theory, Series A},
  volume={175},
  pages={105258},
  year={2020},
  publisher={Elsevier},
  url={https://doi.org/10.1016/j.jcta.2020.105258}
}

@article{aliniaeifard2021extended,
  title={Extended chromatic symmetric functions and equality of ribbon Schur functions},
  author={Aliniaeifard, Farid and Wang, Victor and van Willigenburg, Stephanie},
  journal={Advances in Applied Mathematics},
  volume={128},
  pages={102189},
  year={2021},
  publisher={Elsevier},
  url={https://doi.org/10.1016/j.aam.2021.102189}
}

@unpublished{guay2013modular,
  title={A modular relation for the chromatic symmetric functions of (3+ 1)-free posets},
  author={Guay-Paquet, Mathieu},
  year={2013},
  url={https://arxiv.org/abs/1306.2400}
}

@article{orellana2014graphs,
  title={Graphs with equal chromatic symmetric functions},
  author={Orellana, Rosa and Scott, Geoffrey},
  journal={Discrete Mathematics},
  volume={320},
  pages={1--14},
  year={2014},
  publisher={Elsevier},
  url={https://doi.org/10.1016/j.disc.2013.12.006}
}

@article{crew2020deletion,
  title={A deletion--contraction relation for the chromatic symmetric function},
  author={Crew, Logan and Spirkl, Sophie},
  journal={European Journal of Combinatorics},
  volume={89},
  pages={103143},
  year={2020},
  publisher={Elsevier},
  url={https://doi.org/10.1016/j.ejc.2020.103143}
}

@article{humpert2012incidence,
  title={The incidence Hopf algebra of graphs},
  author={Humpert, Brandon and Martin, Jeremy L},
  journal={SIAM Journal on Discrete Mathematics},
  volume={26},
  number={2},
  pages={555--570},
  year={2012},
  publisher={SIAM},
  url={https://doi.org/10.1137/110820075}
}

@article{takeuchi1971free,
  title={Free Hopf algebras generated by coalgebras},
  author={Takeuchi, Mitsuhiro},
  journal={Journal of the Mathematical Society of Japan},
  volume={23},
  number={4},
  pages={561--582},
  year={1971},
  publisher={The Mathematical Society of Japan},
  url={https://doi.org/10.2969/jmsj/02340561}
}

@article{bernardi2020combinatorial,
  title={Combinatorial reciprocity for the chromatic polynomial and the chromatic symmetric function},
  author={Bernardi, Olivier and Nadeau, Philippe},
  journal={Discrete Mathematics},
  volume={343},
  number={10},
  pages={111989},
  year={2020},
  publisher={Elsevier},
  url={https://doi.org/10.1016/j.disc.2020.111989}
}

@inproceedings{viennot2006heaps,
  title={Heaps of pieces, I: Basic definitions and combinatorial lemmas},
  author={Viennot, G{\'e}rard Xavier},
  booktitle={Combinatoire {\'e}num{\'e}rative: Proceedings of the “Colloque de combinatoire {\'e}num{\'e}rative”, held at Universit{\'e} du Qu{\'e}bec {\`a} Montr{\'e}al, May 28--June 1, 1985},
  pages={321--350},
  year={2006},
  organization={Springer},
  url={https://link.springer.com/chapter/10.1007/BFb0072524}
}

@book{cartier2006problemes,
  title={Probl\'{e}mes combinatoires de commutation et r{\'e}arrangements},
  author={Cartier, Pierre and Foata, Dominique},
  volume={85},
  year={2006},
  publisher={Springer},
  url={https://mat.univie.ac.at/~slc/books/cartfoa.pdf}
}

@article{schmitt1994incidence,
  title={Incidence hopf algebras},
  author={Schmitt, William R},
  journal={Journal of Pure and Applied Algebra},
  volume={96},
  number={3},
  pages={299--330},
  year={1994},
  publisher={Elsevier},
  url={https://doi.org/10.1016/0022-4049(94)90105-8}
}

@unpublished{pierson2025graphs,
  title={On graphs with equal and different Kromatic symmetric functions},
  author={Pierson, Laura and Samanta, Soham},
  year={2025},
  url={https://arxiv.org/abs/2508.17682}
}

@article{lalonde1995lyndon,
  title={Lyndon heaps: an analogue of Lyndon words in free partially commutative monoids},
  author={Lalonde, Pierre},
  journal={Discrete Mathematics},
  volume={145},
  number={1-3},
  pages={171--189},
  year={1995},
  publisher={Elsevier},
  url={https://doi.org/10.1016/0012-365X(94)00032-E}
}

@unpublished{pierson2025power,
  title={Power sum expansions for Kromatic symmetric functions using Lyndon heaps},
  author={Pierson, Laura},
  year={2025},
  url={https://arxiv.org/abs/2502.21285}
}

\end{document}